\documentclass[pdflatex,sn-basic]{sn-jnl}

\usepackage{graphicx}%
\usepackage{multirow}%
\usepackage{amsmath,amssymb,amsfonts}%
\usepackage{amsthm}%
\usepackage{mathrsfs}%
\usepackage[title]{appendix}%
\usepackage{xcolor}%
\usepackage{textcomp}%
\usepackage{manyfoot}%
\usepackage{booktabs}%
\usepackage{algorithm}%
\usepackage{algorithmicx}%
\usepackage{algpseudocode}%
\usepackage{listings}%

\renewcommand{\proofname}{\textbf{Proof}}

\newtheorem{theorem}{Theorem}[section]
\newtheorem{lemma}[theorem]{Lemma}      
\newtheorem{corollary}[theorem]{Corollary}

\newtheorem{definition}[theorem]{Definition}
\newtheorem{example}[theorem]{Example}
\newtheorem{remark}[theorem]{Remark}
\begin{document}

\title[Further results for the dual Hartwig-Spindelb{\"o}ck decomposition and its applications]{Further results for the dual Hartwig-Spindelb{\"o}ck decomposition and its applications}


\author[1]{\fnm{Tan} \sur{Mei}}\email{\url{meitan_0013@163.com}}

\author*[1]{\fnm{Kezheng} \sur{Zuo}}\email{xiangzuo28@163.com}

\author[1]{\fnm{Hui} \sur{Yan}}\email{yanhui@hbnu.edu.cn}

\affil*[1]{Department of Mathematics, School of Mathematics and Statistics, Hubei Normal University, Huangshi 435002, China}


\abstract{In this paper, we introduce two new forms of the dual Hartwig-Spindelb{\"o}ck decomposition and employ them to derive explicit representations for several classes of dual generalized inverses. 
Building on these representations, we further explore and characterize the relationships and properties of these inverses, investigate the dual composite generalized inverses, and verify the applicability of dual partial orders.
The proposed decomposition provides a systematic and convenient framework for the study of dual matrices.}

\keywords{Dual Hartwig-Spindelb{\"o}ck decomposition, DMPGI, DGGI, DCGI, Dual partial orders }



\maketitle

\section{Introduction}

In 1873, Clifford (\citeyear{Clifford}) introduced dual numbers, dual complex numbers and dual quaternions to describe spiral motion, particularly rigid body motion, using quantities with size, direction and position. 
Notably, many problems can first be formulated for spherical motion and then extended to spatial motion through dualization. 
As a result, dual matrices have become an important tool in areas such as robotics and motion control (Chen et al. \citeyear{Chen}; Daniilidis \citeyear{Daniilidis}; Gu and Luh \citeyear{Gu}; Wang et al. \citeyear{Wang4}; Xie et al. \citeyear{Xie}), as well as kinematic analysis and synthesis of spatial mechanisms (de et al. \citeyear{de}; Pennestr{\`\i} et al. \citeyear{Pennestr2}; Udwadia et al. \citeyear{Udwadia1}; Yang and Wang \citeyear{Yang}).

\noindent To better support the theoretical underpinnings of these applications, a deeper understanding of dual matrices is required. 
Motivated by this need, Be and Mishra (\citeyear{Be}) extended the Hartwig-Spindelb{\"o}ck decomposition (H-S decomposition) to the dual number matrices, establishing necessary and sufficient conditions for this decomposition for several classes of matrices.
In this paper, we further present two new forms of this decomposition and derive explicit expressions for various dual generalized inverses, thereby facilitating the study of dual matrices.
Originally introduced by Hartwig and Spindelb{\"o}ck (\citeyear{Hartwig1}) for complex matrices, this decomposition constitutes a powerful tool in matrix analysis (Aghamollaei et al. \citeyear{Aghamollaei}; Baksalary et al. \citeyear{Baksalary1}; Ferreyra and Malik \citeyear{Ferreyra}; Hern{\'a}ndez et al. \citeyear{Hern}; Zuo et al. \citeyear{Zuo}).

\noindent Let us now recall some notations.
\noindent  Dual complex number \(\hat{a} = a_s + \epsilon a_d\) consists of two complex numbers \(a_s\) and \(a_d\), along with the infinitesimal unit \(\epsilon\). Here, \(a_s, a_d \) are called the standard part and the infinitesimal part of \(\hat{a}\), respectively. The unit $\epsilon$ satisfies
$\epsilon \neq 0$, $\epsilon^2 = 0$, $ 0\epsilon = \epsilon 0 = 0$ and $1\epsilon = \epsilon 1 = \epsilon$.
Similarly, with  \(a_s\) and \(a_d\) being real numbers, $\hat{a}$ is a dual real number, i.e., a dual number.
If \( a_s \neq 0 \), then \( \hat{a} \) is called appreciable. It can be shown that \( \hat{a} \) is invertible if and only if it is appreciable, in which case its inverse is given by  $ \hat{a}^{-1}=a_s^{-1}-\epsilon a_s^{-2}a_d$.
For dual numbers \( \hat{a}, \hat{b} \) with  $\hat{a}= a_s + \epsilon a_d, \quad \hat{b} = b_s + \epsilon b_d$, 
where \( a_s, a_d, b_s, b_d \) are real numbers, we define \( \hat{a} \leq \hat{b} \) if either \( a_s < b_s \), or \( a_s = b_s \) and \( a_d \leq b_d \).  
Specifically, \( \hat{a} \) is called positive, nonnegative, nonpositive, or negative if \( \hat{a} > 0 \), \( \hat{a} \geq 0 \), \( \hat{a} \leq 0 \), or \( \hat{a} < 0 \), respectively.
Let \( \mathbb{DC} \), \( \mathbb{C}^{m \times n} \) and \( \mathbb{DC}^{m \times n} \) be the sets of dual complex numbers, \( m \times n \) complex matrices and \( m \times n \) dual complex matrices, respectively. 
Since every entry of a dual complex matrix is a dual complex number, any matrix \( \hat{A} \in \mathbb{DC}^{m \times n} \) can be expressed in the form
\[
\hat{A} = A_s + \epsilon A_d, 
\]
in which $A_{s}, A_{d} \in \mathbb{C}^{m \times n}$ are called the standard part and infinitesimal part of $\hat{A}$, respectively.
If \( A_s \neq 0 \), \( \hat{A} \) is called appreciable.  \( \hat{A} \) is invertible if and only if $ A_s$ is invertible, with inverse  
$\hat{A}^{-1} = A_s^{-1} - \epsilon A_s^{-1} A_d A_s^{-1}$. 
The conjugate transpose of \( \hat{A} \) is  
$\hat{A}^* = A_s^* + \epsilon A_d^*$. 
For \( m = n \), \( \hat{A} \) is a dual unitary  complex matrix if \( \hat{A}^{-1} = \hat{A}^* \).  
For \( \hat{A}, \hat{B} \in \mathbb{DC}^{m \times n} \) with  $\hat{A} = A_s + \epsilon A_d$, $\hat{B} = B_s + \epsilon B_d$, 
where \( A_s, A_d, B_s, B_d  \in \mathbb{C}^{m \times n}\), , we define \( \hat{A} = \hat{B} \) if  \( A_s = B_s \) and \( A_d = B_d \).  
The set of all \( n \times n \)  dual unitary complex matrices is denoted by \( \mathbb{DC}_n^{U} \), and \( I_n \) represents the \( n \times n \) identity matrix.
Moreover, we denote  the appreciable index by $\operatorname{AInd}(\hat{A})$, where $\operatorname{AInd}(\hat{A})=\operatorname{Ind}(A_{s})$.

\noindent  The main purpose of this paper is to present further results for the dual Hartwig-Spindelb{\"o}ck decomposition (D-H-S decomposition) and its applications. Our main contributions can be summarized as follows:
\begin{itemize}
\item We present two new forms of the D-H-S decomposition for dual matrices.
\item We derive explicit expressions for multiple dual generalized inverses through the D-H-S decomposition framework.
\item Based on the D-H-S decomposition, we explore the relationships among various dual generalized inverses, examine composite generalized inverses of dual matrices, and analyze dual partial orders.
\end{itemize}

\noindent The paper is organized as follows.\\
In Section \ref{2}, we provide the necessary lemmas, the definitions of various dual generalized inverses, and the necessary and sufficient conditions for their existence.
In Section \ref{3}, we study two new forms of the D-H-S decomposition for dual matrices. 
In Section \ref{4}, we utilizes the D-H-S decomposition to derive expressions for several dual generalized inverses, examine their interrelations and fundamental properties, and explore their composite forms and applicability to dual partial orders.
Finally, concluding remarks are given in Section \ref{5}.

\section{Preparation}\label{2}
In this section, we will recall some useful results for studying the D-H-S decomposition. 
We begin with the dual singular value decomposition (dual-SVD) for dual matrices, which was given in (Qi and Luo \citeyear{Qi1}).

\begin{lemma}\label{L0}\emph{(Qi and Luo \citeyear{Qi1})}
  Let $\hat{A} \in \mathbb{DC}^{m \times n}$. 
  Then there exist $\hat{U} \in \mathbb{DC}_m^{U}$ and $\hat{V} \in \mathbb{DC}_n^{U}$ such that
\begin{align}\label{A1}
\hat{A} = \hat{U}\begin{bmatrix} \Sigma_0 & 0  \\ 0 & 0 \end{bmatrix}\hat{V}^*,
\end{align}
where 
$$ \Sigma_0 = \mathrm{diag} (\mu_1, \ldots, \mu_r, \mu_{r+1}, \cdots, \mu_{t}),$$ 
$\mu_{1} \geq \mu_{2} \geq \cdots \geq \mu_{r}$ are positive appreciable dual numbers, and $\mu_{r+1} \geq \mu_{r+2} \geq \cdots \geq \mu_{t}$ are positive infinitesimal dual numbers.
Counting possible multiplicities of the diagonal entries, the form $\Sigma_0$ is unique.
\end{lemma}

\begin{remark}\label{RE}
  For the convenience of subsequent computations, $\Sigma_0$  from \emph{Lemma~\ref{L0}} is partitioned as follows:
   \begin{align}\label{MM}
    \begin{bmatrix} \Sigma_0 & 0  \\ 0 & 0 \end{bmatrix} =\begin{bmatrix} \Sigma_1 & 0  \\ 0 & \Sigma_2 \end{bmatrix} = \begin{bmatrix} \Sigma_{\mathrm{1s}} & 0 \\ 0 & 0 \end{bmatrix} + \epsilon\begin{bmatrix} \Sigma_{\mathrm{1d}} & 0 \\ 0 & \Sigma_{\mathrm{2d}} \end{bmatrix}, 
  \end{align}
  with
  \begin{align*}
  &\Sigma_1 = \Sigma_{\mathrm{1s}} +\epsilon \Sigma_{\mathrm{1d}} = \operatorname{diag}(\mu_1, \ldots, \mu_r)\ \  \text{and} \ \
  \Sigma_2 = \epsilon \Sigma_{\mathrm{2d}}= \operatorname{diag}(\mu_{r+1}, \ldots, \mu_{t}, 0, \ldots, 0)\epsilon. 
  \end{align*}
  Here,
  $\Sigma_{\mathrm{1s}}$ and $\Sigma_{\mathrm{1d}}$ are the standard part and infinitesimal part of $\Sigma_1$, $\Sigma_{\mathrm{2d}}$ is the infinitesimal part of $\Sigma_2$, 
  while $\mu_1, \cdots, \mu_r$ are positive appreciable dual numbers and $\mu_{r+1},\cdots,\mu_t$ are positive infinitesimal dual numbers.
\end{remark}

\begin{remark}
 Let $\hat{A}=A_s + \epsilon A_d$. Based on the dual-SVD presented in \emph{Lemma~\ref{L0}}, Cui and Qi \emph{(\citeyear{Cui1})} introduce the dual rank of \(\hat{A}\) as \(\operatorname{Rank}(\hat{A})\) where \(t = \operatorname{Rank}(\hat{A})\), 
 and they further establish its appreciable rank as \(\operatorname{ARank}(\hat{A})\) satisfying \(r = \operatorname{ARank}(\hat{A}) = \operatorname{Rank}(A_s)\).
\end{remark}

\noindent  The essential part of $\hat{A}$ is derived using its dual-SVD.

\begin{definition}\label{L77}\emph{(Cui and Qi \citeyear{Cui1})}
  For $\hat{A} \in \mathbb{DC}^{m \times n}$ with $r = \operatorname{Rank}(\hat{A})$, let its dual-SVD be given as in \emph{Lemma~\ref{L0}} for $\hat{U} \in \mathbb{DC}_m^{U}$ and $\hat{V} \in \mathbb{DC}_n^{U}$, such that
  \begin{align}\label{Evn}
  \hat{A}_e=\hat{U}\begin{bmatrix}
    \Sigma_1 & 0  \\ 0 & 0 
  \end{bmatrix}\hat{V}^*,
   \end{align}
and $\Sigma_1 = \Sigma_{\mathrm{1s}} + \epsilon \Sigma_{\mathrm{1d}} = \operatorname{diag}(\mu_1, \ldots, \mu_r)$ with $\mu_1 \geq \cdots \geq \mu_r$ being positive appreciable dual numbers.
\end{definition}

\noindent We now recall several dual generalized inverses.
\noindent Similar to complex matrices, the definition of the dual Moore-Penrose generalized inverse of a dual matrix is given as follows.
\noindent For $\hat{A} \in \mathbb{DC}^{m \times n}$ and $\hat{X} \in \mathbb{DC}^{n \times m}$, if the two dual matrices satisfy
\begin{align}\label{AN}
\hat{A}\hat{X}\hat{A} = \hat{A}, \quad \hat{X}\hat{A}\hat{X} = \hat{X}, \quad ( \hat{A}\hat{X} )^* = \hat{A}\hat{X}, \quad  ( \hat{X}\hat{A} )^* = \hat{X}\hat{A},
\end{align}
then $\hat{X}$ is called the dual Moore-Penrose generalized inverse (DMPGI) of $\hat{A}$ (Udwadia et al. \citeyear{Udwadia1}), denoted by $\hat{X} = \hat{A}^{\dagger}$.
However, the DMPGI does not exist for all dual matrices, which motivates the need to characterize its existence. 
Wang (\citeyear{Wang2}) established necessary and sufficient conditions for the DMPGI in 2021. 
Subsequently, in 2025, Cui and Qi (\citeyear{Cui1}) provided another characterizations based on dual-SVD.
\begin{lemma}\emph{(Cui and Qi \citeyear{Cui1}; Wang \citeyear{Wang2})} \label{TJ1}
  Let \(\hat{A} = A_s + \epsilon A_d \in \mathbb{DC}^{m \times n}\). Then, the following conditions are equivalent:

\noindent $(a)$ The DMPGI \(\hat{A}^\dagger\) of \(\hat{A}\) exists;

\noindent $(b)$ \((I_m - A_sA_s^\dagger)A_d(I_n - A_s^\dagger A_s) = 0\);

\noindent $(c)$ \(\operatorname{Rank}\begin{bmatrix} A_d & A_s \\ A_s & 0 \end{bmatrix} = 2 \, \operatorname{Rank}(A_s)\);

\noindent $(d)$ $\hat{A}=\hat{A}_e$;

\noindent $(e)$ $\operatorname{ARank}(\hat{A})=\operatorname{Rank}(\hat{A})$.

\noindent  Furthermore, when the DMPGI \(\hat{A}^\dagger\) of \(\hat{A}\) exists,
$$
\hat{A}^\dagger = A_s^\dagger - \epsilon R, 
$$
where \\
$$R = A_s^\dagger A_d A_s^\dagger - (A_s^* A_s)^\dagger A_d^* (I_m - A_s A_s^\dagger) - (I_n - A_s^\dagger A_s) A_d^* (A_s A_s^*)^\dagger.$$
\end{lemma}

\noindent In 2023, Li and Wang (\citeyear{Li1}) introduced the weak dual generalized inverse (WDGI) for $\hat{A} \in \mathbb{DC}^{m \times n}$, defined as the unique dual matrix  $\hat{X} \in \mathbb{DC}^{n \times m}$ satisfying
\begin{align*}
\hat{A}^* \hat{A}\hat{X} \hat{A}\hat{A}^* = \hat{A}^*\hat{A}\hat{A}^*, \quad \hat{X}\hat{A}\hat{X} = \hat{X}, \quad ( \hat{A}\hat{X} )^* = \hat{A}\hat{X}, \quad ( \hat{X}\hat{A} )^* = \hat{X}\hat{A}.
\end{align*}
 However, in 2025, Cui and Qi (\citeyear{Cui1}) introduced the essential part \( \hat{A}_e \) of \( \hat{A} \), whose nonzero singular values correspond to the appreciable singular values of \( \hat{A} \). 
 They subsequently replaced the first equation in (\ref{AN}) with
\[
\hat{A} \hat{X} \hat{A} = \hat{A}_e,
\]
 and defined the unique dual matrix $\hat{X} \in \mathbb{DC}^{n \times m}$ satisfying
\begin{align}\label{E1}
\hat{A} \hat{X} \hat{A} = \hat{A}_e, \quad \hat{X}\hat{A}\hat{X} = \hat{X}, \quad ( \hat{A}\hat{X} )^* = \hat{A}\hat{X}, \quad ( \hat{X}\hat{A} )^* = \hat{X}\hat{A}, 
\end{align}
as the new dual Moore-Penrose inverse (NDMPI) of $\hat{A} \in \mathbb{DC}^{m \times n}$ (Cui and Qi \citeyear{Cui1}), denoted $\hat{A}^N$.
Notably, the NDMPI is equivalent to the WDGI, and Cui and Qi also derive its dual-SVD.

\begin{lemma}\label{LM1}\emph{(Cui and Qi \citeyear{Cui1})}
Let $\hat{A} \in \mathbb{DC}^{m \times n}$ with $r = \operatorname{Rank}(\hat{A})$. Suppose its dual \emph{SVD} follows from \emph{Lemma~\ref{L0}}, then 
\begin{align}\label{NSVD}
  \hat{A}^N=\hat{V}\begin{bmatrix} \Sigma_1^{-1} & 0 \\ 0 & 0 \end{bmatrix}\hat{U}^*,
\end{align}
where $\Sigma_1 = \Sigma_{\mathrm{1s}} + \epsilon \Sigma_{\mathrm{1d}} = \operatorname{diag}(\mu_1, \ldots, \mu_r)$ with $\mu_1 \geq \cdots \geq \mu_r$ being positive appreciable dual numbers.
\end{lemma}

\noindent Wang and Gao (\citeyear{Wang1}) defined a dual matrix with a dual index of 1 (i.e., $\operatorname{Ind}(\hat{A})=1$) as one which satisfies 
$$ \mathcal{R} (\hat{A}^{2})=\mathcal{R} (\hat{A}).$$

\noindent Additional symbols used in this paper are defined by (Wang et al.\citeyear{Wang3}) as the set  $\mathbb{DC}_n^{\mathrm{CM}}$:
\begin{eqnarray*}
\mathbb{DC}_n^{\mathrm{CM}} = \left\{ \hat{A} \in \mathbb{DC}^{n \times n} \middle| \mathcal{R} (\hat{A}^{2})=\mathcal{R} (\hat{A}) \right\} = \left\{ \hat{A} \in \mathbb{DC}^{n \times n} \middle| \operatorname{Ind}(\hat{A})=1\right\}.
\end{eqnarray*}

\noindent In 2022,  the dual group inverse (DGGI) of $\hat{A} \in \mathbb{DC}_n^{\mathrm{CM}}$ (Zhong and Zhang \citeyear{Zhong1}), denoted $\hat{X} = \hat{A}^{\scalebox{0.5}{\#}}$, is the unique dual matrix $\hat{X} \in \mathbb{DC}^{n \times n}$ satisfying:
\begin{align}\label{E3}
\hat{A}\hat{X}\hat{A} = \hat{A}, \quad \hat{X}\hat{A}\hat{X} = \hat{X}, \quad \hat{A}\hat{X} = \hat{X}\hat{A}.
\end{align}

\noindent In 2023,  the dual core generalized inverse (DCGI) of $\hat{A} \in \mathbb{DC}_n^{\mathrm{CM}}$ (Wang and Gao \citeyear{Wang1}), denoted $\hat{X} =\hat{A}^{\tiny\textcircled{\#}}$, is defined to be the unique dual matrix $\hat{X} \in \mathbb{DC}^{n \times n}$ satisfying
\begin{align}\label{En}
 \hat{A}\hat{X}\hat{A} = \hat{A}, \quad \hat{A}\hat{X}^2 = \hat{X}, \quad (\hat{A}\hat{X})^* = \hat{A}\hat{X}.
\end{align}

\noindent Of course, the existence of DGGI and DCGI also depends on certain conditions.
\begin{lemma}\emph{(Zhong and Zhang \citeyear{Zhong1})}
   Let \(\hat{A} = A_s + \epsilon A_d \in \mathbb{DC}^{m \times n}\) with $\operatorname{AInd}(\hat{A})=1$. Then, the following conditions are equivalent:

\noindent $(a)$ The DGGI \(\hat{A}^{\scalebox{0.5}{\#}}\) of \(\hat{A}\) exists;

\noindent $(b)$ \((I_m - A_sA_s^{\scalebox{0.5}{\#}})A_d(I_n - A_sA_s^{\scalebox{0.5}{\#}}) = 0\);

\noindent $(c)$ \(\operatorname{Rank}\begin{bmatrix} A_d & A_s \\ A_s & 0 \end{bmatrix} = 2 \, \operatorname{Rank}(A_s)\);

\noindent $(d)$ $\hat{A}^{\dagger}$ exists.\\
Furthermore, if the DGGI $\hat{A}^{\scalebox{0.5}{\#}}$ exists, then
\[
\hat{A}^{\scalebox{0.5}{\#}} = A_s^{\scalebox{0.5}{\#}} - \epsilon R,
\]
where
\[
R = A_s^{\scalebox{0.5}{\#}} A_d A_s^{\scalebox{0.5}{\#}} - \left(A_s^{\scalebox{0.5}{\#}}\right)^2 A_d \left(I - A_s A_s^{\scalebox{0.5}{\#}}\right) - \left(I - A_s A_s^{\scalebox{0.5}{\#}}\right) A_d \left(A_s^{\scalebox{0.5}{\#}}\right)^2.
\]
\end{lemma}

\begin{lemma}\emph{(Wang and Gao \citeyear{Wang1})}
   Let \(\hat{A} = A_s + \epsilon A_d \in \mathbb{DC}^{m \times n}\) with $\operatorname{AInd}(\hat{A})=1$. Then, the following conditions are equivalent:

\noindent $(a)$ The DCGI \(\hat{A}^{\tiny\textcircled{\#}}\) of \(\hat{A}\) exists;

\noindent $(b)$ \((I_m - A_sA_s^{\tiny\textcircled{\#}})A_d(I_n - A_s^{\tiny\textcircled{\#}}A_s) = 0\).\\
Furthermore, if the DCGI $\hat{A}^{\tiny\textcircled{\#}}$ exists, then
\[
\hat{A}^{\tiny\textcircled{\#}} = A_s^{\tiny\textcircled{\#}} - \epsilon R,
\]
where
\[  
R= A_s^{\tiny\textcircled{\#}} A_d A_s^{\dagger} - A_s^{\scalebox{0.5}{\#}} A_d A_s^{\dagger} + A_s^{\scalebox{0.5}{\#}} A_d A_s^{\tiny\textcircled{\#}} - A_s^{\tiny\textcircled{\#}} (A_d A_s^{\dagger})^* (I_{n} - A_s A_s^{\dagger}) - (I_{n} - A_s A_s^{\scalebox{0.5}{\#}}) A_d A_s^{\scalebox{0.5}{\#}} A_s^{\tiny\textcircled{\#}}.
\]
\end{lemma}

\section{D-H-S decomposition}\label{3}
First, we use the dual-SVD to derive the D-H-S decomposition.
\begin{theorem}\label{LL1}
  Let $\hat{A} \in \mathbb{DC}^{n \times n}$ with $t=\operatorname{Rank}(\hat{A})$.
Then there exist $\hat{U} \in \mathbb{DC}_n^{U}$ such that
\begin{align}\label{F0}
 \hat{A}  &=\hat{U}\begin{bmatrix}
\Sigma_0 K_0 & \Sigma_0 L_0\\
0 & 0
\end{bmatrix}\hat{U}^*,
\end{align}
where $K_0 \in \mathbb{DC}^{t \times t}$, $L_0 \in \mathbb{DC}^{t \times (n-t)}$ satisfying $K_0K_0^*+L_0L_0^*=I_t$,
and $$ \Sigma_0 = \mathrm{diag} (\mu_1, \ldots, \mu_r, \mu_{r+1}, \cdots, \mu_{t}),$$ 
with
$\mu_{1} \geq \mu_{2} \geq \cdots \geq \mu_{r}$ being positive appreciable dual numbers, $\mu_{r+1} \geq \mu_{r+2} \geq \cdots \geq \mu_{t}$ being positive infinitesimal dual numbers.
\end{theorem}
\begin{proof}
  \noindent  For $\hat{A} \in \mathbb{DC}^{n \times n}$, Lemma~\ref{L0} implies $\hat{A}$ has the form given in equation (\ref{A1}).
\noindent Suppose $$
\hat{V}^*\hat{U}=\begin{bmatrix}
  K_0 & L_0 \\ M_0 & N_0
\end{bmatrix} \in \mathbb{DC}_n^{U},$$ 
where $K_0 \in \mathbb{DC}^{t \times t}$, $L_0 \in \mathbb{DC}^{t \times (n-t)}$.
Thus,
\begin{align*}
 \hat{A} = \hat{U} \begin{bmatrix} 
 \Sigma_0 & 0 \\ 0 & 0
\end{bmatrix} \hat{V}^* 
 &= \hat{U}\begin{bmatrix} 
 \Sigma_0 & 0 \\ 0 & 0
\end{bmatrix}\begin{bmatrix} 
  K_0 & L_0 \\ M_0 & N_0
\end{bmatrix} \hat{U}^*
=\hat{U}\begin{bmatrix}
\Sigma_0 K_0 & \Sigma_0 L_0\\
0 & 0
\end{bmatrix}\hat{U}^*.
\end{align*}
Furthermore, $(\hat{V}^*\hat{U})(\hat{V}^*\hat{U})^*=I_n$ leads directly to $K_0K_0^*+L_0L_0^*=I_t$.
\end{proof}

\noindent Recently, Be and Mishra (\citeyear{Be}) formulated a version of the D-H-S decomposition. To enhance its utility for subsequent research, we have further supplemented it with a relational expression concerning its dual submatrix.

\begin{lemma}\label{CS}\emph{(Be and Mishra \citeyear{Be})}
  Let $\hat{A} \in \mathbb{DC}^{n \times n}$ with $r=\operatorname{ARank}(\hat{A})$ and $t=\operatorname{Rank}(\hat{A})$.
Then there exist $\hat{U} \in \mathbb{DC}_n^{U}$ such that
\begin{align}
 \hat{A} &=\hat{U}\begin{bmatrix}
\Sigma_1 K & \Sigma_1 L\\
\Sigma_2 M & \Sigma_2 N
\end{bmatrix}\hat{U}^*, \label{F1}
\end{align}
where $K \in \mathbb{DC}^{r \times r}$, $L \in \mathbb{DC}^{r \times (n-r)}$, $M\in \mathbb{DC}^{(n-r)\times r }$ and $N\in \mathbb{DC}^{(n-r)\times (n-r) }$ satisfying 
\begin{align}\label{SS}
KK^*+LL^*=I_r, \quad KM^*+LN^*=0, \quad MM^*+NN^*=I_{n-r},
\end{align}
and $\Sigma_1 = \Sigma_{\mathrm{1s}} + \epsilon \Sigma_{\mathrm{1d}} = \operatorname{diag}(\mu_1, \ldots, \mu_r)$ with $\mu_1 \geq \cdots \geq \mu_r$ being positive appreciable dual numbers,
 $\Sigma_2 = \epsilon \Sigma_{\mathrm{2d}} = \operatorname{diag}(\mu_{r+1}, \ldots, \mu_t)$ with $\mu_{r+1} \geq \cdots \geq \mu_t$ being positive infinitesimal dual numbers.
\end{lemma}

\noindent Due to the fact that the size of two dual matrices is determined by comparing their standard parts and infinitesimal parts separately, 
the D-H-S decomposition is further refined to extend its scope to dual matrices.

\begin{theorem}\label{L1}
Let $\hat{A} \in \mathbb{DC}^{n \times n}$ with $r=\operatorname{ARank}(\hat{A})$ and $t=\operatorname{Rank}(\hat{A})$.
Then there exist $\hat{U} \in \mathbb{DC}_n^{U}$ such that
\begin{align}
 \hat{A} 
&= \hat{U} \left(\begin{bmatrix}
\Sigma_1 K_{1} & \Sigma_1 L_{1} \\
\Sigma_2 M_{1} & \Sigma_2  N_{1}
\end{bmatrix} +  \epsilon \begin{bmatrix}
\Sigma_1 K_{2} & \Sigma_1 L_{2} \\
\Sigma_2 M_{2} & \Sigma_2  N_{2}
\end{bmatrix}\right)\hat{U}^* \notag\\
&= \hat{U} \left(\begin{bmatrix}
\Sigma_{\mathrm{1s}} K_{1} & \Sigma_{\mathrm{1s}} L_{1} \\ 0 & 0
\end{bmatrix} +\epsilon \begin{bmatrix}
\Sigma_{\mathrm{1d}} K_{1}+\Sigma_{\mathrm{1s}} K_{2} & \Sigma_{\mathrm{1d}} L_{1}+\Sigma_{\mathrm{1s}} L_{2} \\
\Sigma_{\mathrm{2d}} M_{1} & \Sigma_{\mathrm{2d}}  N_{1}
\end{bmatrix}\right)\hat{U}^* \label{ab},
\end{align}
with
  \begin{align*}
  &\Sigma_1 = \Sigma_{\mathrm{1s}} +\epsilon \Sigma_{\mathrm{1d}} = \operatorname{diag}(\mu_1, \ldots, \mu_r)\ \  \text{and} \ \
  \Sigma_2 = \epsilon \Sigma_{\mathrm{2d}}= \operatorname{diag}(\mu_{r+1}, \ldots, \mu_{t}, 0, \ldots, 0)\epsilon,
  \end{align*}
where $\mu_1, \cdots, \mu_r$ are positive appreciable dual numbers, $\mu_{r+1},\cdots,\mu_t$ are positive infinitesimal dual numbers,
and $$K_1,K_2 \in \mathbb{C}^{r \times r},\ \ L_1,L_2 \in \mathbb{C}^{r \times (n-r)},\ \ M_1\in \mathbb{C}^{(n-r)\times r },\ \ N_1\in \mathbb{C}^{(n-r)\times (n-r) }$$ satisfying
\begin{align}
 &K_1 K_1^*+L_1L_1^*=I_r \quad \text{and} \quad K_1 K_2^* +K_2 K_1^*+L_1 L_2^*+L_2 L_1^*=0,\label{S1} \\
&M_1 M_1^*+N_1N_1^*=I_{n-r} \quad \text{and} \quad K_1 M_1^* +L_1 N_1^*=0.\label{S2}
\end{align}
\end{theorem}
\begin{proof}
\noindent The dual matrix in (\ref{F1}) is partitioned as $$
\begin{bmatrix}
  K & L \\ M & N
\end{bmatrix} 
=  \begin{bmatrix}
   K_1 & L_1 \\ M_1 & N_1 
\end{bmatrix} + \epsilon  \begin{bmatrix}
  K_2 & L_2 \\ M_2 & N_2  
\end{bmatrix},
$$
 with $K_1, K_2 \in \mathbb{C}^{r \times r}$, $L_1, L_2 \in \mathbb{C}^{r \times (n-r)}$, $M_1, M_2\in \mathbb{C}^{(n-r)\times r }$ and $N_1, N_2\in \mathbb{C}^{(n-r)\times (n-r) }$.
From (\ref{SS}), we directly obtain (\ref{S1}) and (\ref{S2}). Furthermore, we have
\begin{align*}
\hat{A}&=\hat{U}\begin{bmatrix}
\Sigma_1 K & \Sigma_1 L\\
\Sigma_2 M & \Sigma_2 N
\end{bmatrix}\hat{U}^*\\ 
&=\hat{U} \begin{bmatrix} 
  \Sigma_1(K_1+\epsilon K_2) & \Sigma_1(L_1+\epsilon L_2)  \\
   \Sigma_2(M_1+\epsilon M_2) & \Sigma_2(N_1+\epsilon N_2)
\end{bmatrix} \hat{U}^*\\
&=\hat{U} \left(\begin{bmatrix}
\Sigma_1 K_{1} & \Sigma_1 L_{1} \\
\Sigma_2 M_{1} & \Sigma_2  N_{1}
\end{bmatrix} +  \epsilon \begin{bmatrix}
\Sigma_1 K_{2} & \Sigma_1 L_{2} \\
\Sigma_2 M_{2} & \Sigma_2  N_{2}
\end{bmatrix}\right)\hat{U}^*,
\end{align*}
which then implies 
\begin{align*}
\hat{A}&=\hat{U} \left(\begin{bmatrix}
(\Sigma_{\mathrm{1s}} +\epsilon \Sigma_{\mathrm{1d}}) K_{1} & (\Sigma_{\mathrm{1s}} +\epsilon \Sigma_{\mathrm{1d}}) L_{1} \\
\epsilon \Sigma_{\mathrm{2d}} M_{1} & \epsilon \Sigma_{\mathrm{2d}}  N_{1}
\end{bmatrix} +  \epsilon \begin{bmatrix}
(\Sigma_{\mathrm{1s}} +\epsilon \Sigma_{\mathrm{1d}}) K_{2} & (\Sigma_{\mathrm{1s}} +\epsilon \Sigma_{\mathrm{1d}}) L_{2} \\
\epsilon \Sigma_{\mathrm{2d}} M_{2} & \epsilon \Sigma_{\mathrm{2d}}  N_{2}
\end{bmatrix}\right)\hat{U}^*\\
&=\hat{U} \left(\begin{bmatrix}
\Sigma_{\mathrm{1s}} K_{1} & \Sigma_{\mathrm{1s}} L_{1} \\ 0 & 0
\end{bmatrix} +\epsilon \begin{bmatrix}
\Sigma_{\mathrm{1d}} K_{1}+\Sigma_{\mathrm{1s}} K_{2} & \Sigma_{\mathrm{1d}} L_{1}+\Sigma_{\mathrm{1s}} L_{2} \\
\Sigma_{\mathrm{2d}} M_{1} & \Sigma_{\mathrm{2d}}  N_{1}
\end{bmatrix}\right)\hat{U}^*.
\end{align*}
\end{proof}

\begin{remark}\label{RR}
\noindent By \emph{(\ref{Evn})}, we have
\begin{align*}
\hat{A}_e  &= \hat{U} \begin{bmatrix} \Sigma_1 & 0 \\ 0 & 0 \end{bmatrix} \hat{V}^*
= \hat{U} \begin{bmatrix} \Sigma_1 & 0 \\ 0 & 0 \end{bmatrix} \hat{V}^*\hat{U}\hat{U}^*
=\hat{U} \begin{bmatrix} \Sigma_1 & 0 \\ 0 & 0 \end{bmatrix} \begin{bmatrix}
   K & L \\ M & N 
\end{bmatrix}\hat{U}^*\\
&= \hat{U}\begin{bmatrix} \Sigma_1 & 0 \\ 0 & 0 \end{bmatrix}\left( \begin{bmatrix}
   K_1 & L_1 \\ M_1 & N_1 
\end{bmatrix} + \epsilon \begin{bmatrix}
  K_2 & L_2 \\ M_2 & N_2  
\end{bmatrix} \right)\hat{U}^*\\
&=\hat{U} \left(\begin{bmatrix}
\Sigma_1 K_{1} & \Sigma_1 L_{1} \\
0 & 0
\end{bmatrix} + \epsilon \begin{bmatrix}
\Sigma_{\mathrm{1s}} K_{2} & \Sigma_{\mathrm{1s}} L_{2} \\
0 & 0
\end{bmatrix}\right) \hat{U}^*,
\end{align*}
where $K_1,K_2 \in \mathbb{C}^{r \times r}$, $L_1,L_2 \in \mathbb{C}^{r \times (n-r)}$  satisfying
 $$ K_1 K_1^*+L_1L_1^*=I_r, \quad K_1 K_2^* +K_2 K_1^*+L_1 L_2^*+L_2 L_1^*=0,$$
and $\Sigma_1 = \Sigma_{\mathrm{1s}} + \epsilon \Sigma_{\mathrm{1d}} = \operatorname{diag}(\mu_1, \ldots, \mu_r)$ with $\mu_1 \geq \cdots \geq \mu_r$ being positive appreciable dual numbers.
\end{remark}

\noindent From Remark~\ref{RR}, $\hat{A} = \hat{A}_{e}$ if and only if $\Sigma_{2\mathrm{d}} = 0$. 
Therefore, by recalling (\ref{AN}) and (\ref{E1}), this observation directly confirms Theorem 5.2 of (Cui and Qi \citeyear{Cui1}), which states that $\hat{A}^{N} = \hat{A}^{\dagger}$ under the same condition.
\noindent The following two examples demonstrate the D-H-S decomposition, as established by Theorem~\ref{LL1} and Lemma~\ref{CS}.

\begin{example}\label{ZZ2}
  Let 
\begin{align*}
  \hat{A} = \begin{bmatrix}
1 + \epsilon & -3\epsilon & 0 & -3\epsilon \\
3\epsilon & 2 + \epsilon & -4\epsilon & -2\epsilon \\
-\epsilon & 2\epsilon & 3\epsilon & 0 \\
\epsilon & -2\epsilon & 0 & 0
\end{bmatrix},
\end{align*}
then there exists $\hat{U}=\begin{bmatrix}
1 & -\epsilon & \epsilon & -\epsilon\\
\epsilon & 1 & -\epsilon & \epsilon\\
-\epsilon & \epsilon & 1 & -\epsilon\\
\epsilon & -\epsilon & \epsilon &1
\end{bmatrix} \in \mathbb{DC}_4^{U}$, 
$\hat{V}=\begin{bmatrix}
1 & \epsilon & 0 & 3\epsilon\\
-\epsilon & 1 & 2\epsilon  & \epsilon\\
0 & -2\epsilon & 1 & \epsilon\\
-3\epsilon & -\epsilon & -\epsilon & 1
\end{bmatrix} \in \mathbb{DC}_4^{U}$ such that
$$\Sigma = \hat{U}^*\hat{A}\hat{V}=\begin{bmatrix}
1 + \epsilon & 0 & 0 & 0\\
0 & 2 + \epsilon & 0 & 0\\
0 & 0 & 3\epsilon & 0 \\
0 & 0 & 0 & 0 
\end{bmatrix},\ \ \text{i.e.,} \ \  \Sigma_0= \begin{bmatrix}
  1 + \epsilon & 0  &0 \\
0 & 2 + \epsilon & 0\\
0 & 0 & 3\epsilon
\end{bmatrix}.$$
It follows that 
$$\hat{A}=\hat{U}\begin{bmatrix}
\Sigma_0 K_0 & \Sigma_0 L_0\\
0 & 0
\end{bmatrix}\hat{U}^*=\hat{U}
\begin{bmatrix}
1+\epsilon & -2\epsilon & \epsilon & -4\epsilon \\
4\epsilon & 2+\epsilon & -6\epsilon & 0 \\
0 & 0 & 3\epsilon & 0 \\
0 & 0 & 0 & 0
\end{bmatrix} \hat{U}^*,$$
with $K_0=\begin{bmatrix}
  1 & -2\epsilon & \epsilon \\
  2\epsilon & 1 & -3\epsilon \\
-\epsilon & 3\epsilon & 1
\end{bmatrix}$  and $L_0=\begin{bmatrix}
-4\epsilon \\ 0 \\ -2\epsilon
\end{bmatrix}$ 
satisfying  $K_0K_0^*+L_0L_0^*=I_3$.
\end{example}

\begin{example}\label{ZZ1}
  Let 
  $$\hat{A} = 
\begin{bmatrix}
1+\epsilon & -3\epsilon & 0 & -3\epsilon \\
3\epsilon & 2+\epsilon & -4\epsilon & -2\epsilon \\
-\epsilon & 2\epsilon & 0 & 0 \\
\epsilon & -2\epsilon & 0 & 0
\end{bmatrix},$$
then then there exists $\hat{U},\hat{V}$ of the same form as in \emph{Example~\ref{ZZ2}} such that
$$ \Sigma = \begin{bmatrix}
1 + \epsilon & 0 & 0 & 0\\
0 & 2 + \epsilon & 0 & 0 \\
0 & 0 & 0 & 0 \\
0 & 0 & 0 & 0
\end{bmatrix},
\ \ \text{i.e.,} \ \  \Sigma_1= \begin{bmatrix}
  1 + \epsilon & 0  \\
0 & 2 + \epsilon
\end{bmatrix},\ \  \Sigma_2= \begin{bmatrix} 
  0 & 0 \\ 0 & 0
  \end{bmatrix}.$$
Therefore,
$$\hat{A}=\hat{U}\begin{bmatrix}
\Sigma_1 K & \Sigma_1 L\\
\Sigma_2 M & \Sigma_2 N
\end{bmatrix}\hat{U}^*=\hat{U}
\begin{bmatrix}
1+\epsilon & -2\epsilon & \epsilon & -4\epsilon \\
4\epsilon & 2+\epsilon & -6\epsilon & 0 \\
0 & 0 & 0 & 0 \\
0 & 0 & 0 & 0
\end{bmatrix} \hat{U}^*,$$
where $K= \begin{bmatrix} 1 & -2\epsilon \\ 2\epsilon & 1 \end{bmatrix}$,
$L= \begin{bmatrix} \epsilon & -4\epsilon \\ -3\epsilon & 0 \end{bmatrix}$,
$M=\begin{bmatrix} -\epsilon & 3\epsilon \\ 4\epsilon & 0 \end{bmatrix}$ and
$N=\begin{bmatrix} 1 & -2\epsilon \\ 2\epsilon & 1 \end{bmatrix}$
satisfy \emph{(\ref{S1})} and \emph{(\ref{S2})}.
\end{example}

\section{Applications of the D-H-S Decomposition}\label{4}
In this section, we first use the D–H–S decomposition to derive expressions for the DMPGI, NDMPI, DGGI, and DCGI, and then apply these results to study dual generalized inverses.

\begin{lemma}\label{TG1}\emph{(Be and Mishra \citeyear{Be})}
  Let $\hat{A} \in \mathbb{DC}^{n \times n}$ with $r=\operatorname{Rank}(\hat{A})$.
Then there exist $\hat{U} \in \mathbb{DC}_n^{U}$ such that
\begin{align*}
\hat{A}^{N} &=\hat{U}\begin{bmatrix}
  K^* \Sigma_1^{-1}  &  0 \\ L^* \Sigma_1^{-1} & 0
\end{bmatrix}\hat{U}^*,
\end{align*}
where $K \in \mathbb{DC}^{r \times r}$, $L \in \mathbb{DC}^{r \times (n-r)}$ satisfying $KK^*+LL^*=I_r$, 
and $\Sigma_1 = \Sigma_{\mathrm{1s}} + \epsilon \Sigma_{\mathrm{1d}} = \operatorname{diag}(\mu_1, \ldots, \mu_r)$ with $\mu_1 \geq \cdots \geq \mu_r$ being positive appreciable dual numbers.
\end{lemma}

\begin{theorem}\label{TG}
  Let $\hat{A} \in \mathbb{DC}^{n \times n}$ with $r=\operatorname{Rank}(\hat{A})$.
Suppose its dual-SVD is derived from \emph{Lemma~\ref{LM1}} for $\hat{U} \in \mathbb{DQ}_n^{U}$ such that
\begin{align}
\hat{A}^{N} 
&=\hat{U}\left(\begin{bmatrix}
K_{1}^* \Sigma_1^{-1}  &  0 \\ L_{1}^* \Sigma_1^{-1} & 0
\end{bmatrix} + \begin{bmatrix}
K_{2}^*\Sigma_1^{-1} & 0  \\ L_{2}^*\Sigma_1^{-1} & 0
\end{bmatrix} \epsilon\right)\hat{U}^* \notag\\ 
&=  \hat{U}\left(\begin{bmatrix}
K_{1}^* \Sigma_{\mathrm{1s}}^{-1}  &  0 \\
L_{1}^* \Sigma_{\mathrm{1s}}^{-1} & 0
\end{bmatrix} + \begin{bmatrix}
K_{2}^*\Sigma_{\mathrm{1s}}^{-1}-K_{1}^* \Sigma_{\mathrm{1s}}^{-2}\Sigma_{\mathrm{1d}}   & 0  \\
L_{2}^*\Sigma_{\mathrm{1s}}^{-1}-L_{1}^* \Sigma_{\mathrm{1s}}^{-2}\Sigma_{\mathrm{1d}}   & 0
\end{bmatrix} \epsilon\right)\hat{U}^*, \notag
\end{align} 
where $K_1,K_2 \in \mathbb{C}^{r \times r}$, $L_1,L_2 \in \mathbb{C}^{r \times (n-r)}$  satisfying
 $$ K_1 K_1^*+L_1L_1^*=I_r, \quad K_1 K_2^* +K_2 K_1^*+L_1 L_2^*+L_2 L_1^*=0,$$
and $\Sigma_1 = \Sigma_{\mathrm{1s}} + \epsilon \Sigma_{\mathrm{1d}} = \operatorname{diag}(\mu_1, \ldots, \mu_r)$ with $\mu_1 \geq \cdots \geq \mu_r$ being positive appreciable dual numbers.
\end{theorem}
\begin{proof}
Following Lemma~\ref{TG1}, we assume that
 $$ K=K_1+\epsilon K_2 \in \mathbb{DC}^{r \times r}, \quad L=L_1+\epsilon L_2 \in \mathbb{DC}^{r \times (n-r)},$$
where $K_1,K_2 \in \mathbb{C}^{r \times r}$ and $L_1,L_2 \in \mathbb{C}^{r \times (n-r)}$.
Since $KK^*+LL^*=I_r$, it follows that 
$$ K_1 K_1^*+L_1L_1^*=I_r, \quad K_1 K_2^* +K_2 K_1^*+L_1 L_2^*+L_2 L_1^*=0.$$
Consequently, we have
\begin{align*}
 \hat{A}^N &=\hat{U}\begin{bmatrix}
  K^* \Sigma_1^{-1}  &  0 \\ L^* \Sigma_1^{-1} & 0
\end{bmatrix}\hat{U}^*
=\hat{U}\begin{bmatrix}
   (K_1^*+\epsilon K_2^*)\Sigma_1^{-1} & 0 \\
 (L_1^*+\epsilon L_2^*)\Sigma_1^{-1} & 0 
\end{bmatrix} \hat{U}^*
= \hat{U}\left(\begin{bmatrix}
K_{1}^* \Sigma_1^{-1}  &  0 \\ L_{1}^* \Sigma_1^{-1} & 0
\end{bmatrix} + \epsilon \begin{bmatrix}
K_{2}^*\Sigma_1^{-1} & 0  \\ L_{2}^*\Sigma_1^{-1} & 0
\end{bmatrix}\right)\hat{U}^*,
\end{align*}
which further leads to
\begin{align*}
 \hat{A}^N &= \hat{U}\left(\begin{bmatrix}
K_{1}^* (\Sigma_{\mathrm{1s}}^{-1}-\epsilon\Sigma_{\mathrm{1s}}^{-2}\Sigma_{\mathrm{1d}})  &  0 \\ L_{1}^* (\Sigma_{\mathrm{1s}}^{-1}-\epsilon\Sigma_{\mathrm{1s}}^{-2}\Sigma_{\mathrm{1d}}) & 0
\end{bmatrix} + \epsilon\begin{bmatrix}
  K_{2}^*(\Sigma_{\mathrm{1s}}^{-1}-\Sigma_{\mathrm{1s}}^{-2}\Sigma_{\mathrm{1d}}\epsilon)  & 0  \\ L_{2}^*(\Sigma_{\mathrm{1s}}^{-1}-\Sigma_{\mathrm{1s}}^{-2}\Sigma_{\mathrm{1d}}\epsilon) & 0
\end{bmatrix} \right)\hat{U}^*\\
 &= \hat{U}\left(\begin{bmatrix}
K_{1}^* \Sigma_{\mathrm{1s}}^{-1}  &  0 \\
L_{1}^* \Sigma_{\mathrm{1s}}^{-1} & 0
\end{bmatrix} + \begin{bmatrix}
K_{2}^*\Sigma_{\mathrm{1s}}^{-1}-K_{1}^* \Sigma_{\mathrm{1s}}^{-2}\Sigma_{\mathrm{1d}}   & 0  \\
L_{2}^*\Sigma_{\mathrm{1s}}^{-1}-L_{1}^* \Sigma_{\mathrm{1s}}^{-2}\Sigma_{\mathrm{1d}}   & 0
\end{bmatrix} \epsilon\right)\hat{U}^*.
\end{align*}
\end{proof}

\begin{corollary}\label{C1}
  Let \( \hat{A} \in \mathbb{DC}^{n \times n} \) with $r=\operatorname{Rank}(\hat{A})$. Then the DMPGI $\hat{A}^{\dagger}$ exists, and exists $\hat{U} \in \mathbb{DC}_n^{U}$ such that
\begin{align}
\hat{A}^{\dagger} &=\hat{U}\begin{bmatrix}
  K^* \Sigma_1^{-1}  &  0 \\ L^* \Sigma_1^{-1} & 0
\end{bmatrix}\hat{U}^* \label{MP4}\\
&= \hat{U}\left(\begin{bmatrix}
K_{1}^* \Sigma_1^{-1}  &  0 \\
L_{1}^* \Sigma_1^{-1} & 0
\end{bmatrix} + \epsilon \begin{bmatrix}
K_{2}^*\Sigma_1^{-1} & 0  \\
L_{2}^*\Sigma_1^{-1} & 0
\end{bmatrix} \right)\hat{U}^*  \notag\\ 
&=\hat{U}\left(\begin{bmatrix}
K_{1}^* \Sigma_{\mathrm{1s}}^{-1}  &  0 \\
L_{1}^* \Sigma_{\mathrm{1s}}^{-1} & 0
\end{bmatrix} +  \epsilon \begin{bmatrix}
 K_{2}^*\Sigma_{\mathrm{1s}}^{-1} -K_{1}^* \Sigma_{\mathrm{1s}}^{-2}\Sigma_{\mathrm{1d}}  & 0  \\
L_{2}^*\Sigma_{\mathrm{1s}}^{-1}-L_{1}^* \Sigma_{\mathrm{1s}}^{-2}\Sigma_{\mathrm{1d}}  & 0
\end{bmatrix}\right)\hat{U}^*,  \notag
\end{align}
where $K=K_1+\epsilon K_2 \in \mathbb{DC}^{r \times r}$, $L=L_1+\epsilon L_2 \in \mathbb{DC}^{r \times (n-r)}$ satisfy 
$KK^* + LL^* = I_r$, 
and $\Sigma_1 = \Sigma_{\mathrm{1s}} + \epsilon \Sigma_{\mathrm{1d}} = \operatorname{diag}(\mu_1, \ldots, \mu_r)$ with $\mu_1 \geq \cdots \geq \mu_r$ being positive appreciable dual numbers.
\end{corollary}
\begin{proof}
As established by (Cui and Qi \citeyear{Cui1}) and (Li and Wang \citeyear{Li1}), for the form of  $\hat{A} $ in Lemma~\ref{CS}, its DMPGI $\hat{A}^{\dagger}$ exists and  satisfies $\hat{A}^{\dagger}=\hat{A}^N$
when $\Sigma_2=0$ holds.
This fact, combined with Lemma~\ref{TG1} and Theorem~\ref{TG}, completes the proof. It is therefore omitted.
\end{proof}

\noindent According to the above theorems, whether $\Sigma_2 = 0$ determines the existence of the DMPGI. Moreover, if $\Sigma_2 = 0$, the DMPGI equals the NDMPI. 
Below, we further illustrate this point using Example~\ref{ZZ2} and Example~\ref{ZZ1}, and apply Theorem~\ref{TG} and Corollary~\ref{C1} to compute the DMPGI and NDMPI.

\begin{example}
 For  $\hat{A}$ in \emph{Example~\ref{ZZ2}},  we have 
  $$\Sigma_1=\mathrm{diag} (1+\epsilon, \ 2+\epsilon) \ \ \text{and} \ \ \Sigma_2=\mathrm{diag} (3\epsilon,\ 0 ).$$
 When $\Sigma_2\neq 0$, $\hat{A}^{\dagger}$ does not exist, and \emph{Lemma~\ref{TG1}} yields
$$\hat{A}^{N} =\hat{U}\begin{bmatrix}
  K^* \Sigma_1^{-1}  &  0 \\ L^* \Sigma_1^{-1} & 0
\end{bmatrix}\hat{U}^*= \hat{U}\begin{bmatrix}
1-\epsilon & \epsilon & 0 & 0\\
-2\epsilon & \frac{1}{2}-\frac{\epsilon}{4} & 0 & 0\\
\epsilon & -\frac{3\epsilon}{2} & 0 & 0 \\
-4\epsilon & 0 & 0 & 0 
\end{bmatrix}\hat{U}^*
=\begin{bmatrix}
1-\epsilon & \frac{3\epsilon}{2} & -\epsilon & \epsilon \\
-\frac{3\epsilon}{2} & \frac{1}{2}-\frac{\epsilon}{4} & \frac{\epsilon}{2} & -\frac{\epsilon}{2} \\
0 & -\epsilon & 0 & 0 \\
-3\epsilon & -\frac{\epsilon}{2} & 0 & 0
\end{bmatrix}.$$
\end{example}

\begin{example}
  Let $\hat{A}$ be \emph{Example~\ref{ZZ1}}. Then
  $$\Sigma_1=\mathrm{diag} (1+\epsilon, \ 2+\epsilon) \ \ \text{and} \ \ \Sigma_2=\mathrm{diag} (0,\ 0 ).$$
  For the case $\Sigma_2=0$, $\hat{A}^{\dagger}$ exists.
  Then by \emph{Lemma~\ref{TG1}} and \emph{Corollary~\ref{C1}}, we have
\begin{align*}
\hat{A}^{\dagger}=\hat{A}^{N} =\hat{U}\begin{bmatrix}
  K^* \Sigma_1^{-1}  &  0 \\ L^* \Sigma_1^{-1} & 0
\end{bmatrix}\hat{U}^*=\hat{U}\begin{bmatrix}
1-\epsilon & \epsilon & 0 & 0\\
-2\epsilon & \frac{1}{2}-\frac{\epsilon}{4} & 0 & 0\\
\epsilon & -\frac{3\epsilon}{2} & 0 & 0 \\
-4\epsilon & 0 & 0 & 0 
\end{bmatrix}\hat{U}^*
=\begin{bmatrix}
1-\epsilon & \frac{3\epsilon}{2} & -\epsilon & \epsilon \\
-\frac{3\epsilon}{2} & \frac{1}{2}-\frac{\epsilon}{4} & \frac{\epsilon}{2} & -\frac{\epsilon}{2} \\
0 & -\epsilon & 0 & 0 \\
-3\epsilon & -\frac{\epsilon}{2} & 0 & 0
\end{bmatrix}.
\end{align*}
\end{example}

\noindent Below we state the necessary and sufficient condition for the dual index of a dual matrix to be 1.
\begin{theorem}\label{J1}
  Let $\hat{A} \in \mathbb{DC}^{n\times n}$ have the forms  in \emph{(\ref{F0})} and \emph{(\ref{F1})} of \emph{Theorem~\ref{LL1}} and \emph{Lemma~\ref{CS}}, respectively.  Then the following statements are equivalent:\\
\noindent $(a)$ $\operatorname{Ind}(\hat{A})=1$; \\
\noindent $(b)$ $\operatorname{AInd}(\hat{A})=1$ and $\hat{A}^{\dagger}$ exists;\\
\noindent $(c)$ $\hat{A}^{\scalebox{0.5}{\#}}$ exists;\\
\noindent $(d)$ $\hat{A}^{\tiny\textcircled{\#}}$ exists;\\
\noindent $(e)$ $K_1$ is invertible and $\Sigma_{\mathrm{2d}}=0$;\\
\noindent $(f)$ $K$ is invertible and $\Sigma_{\mathrm{2d}}=0$.
\end{theorem}
\begin{proof}
\noindent  We only need to prove $(b) \Leftrightarrow  (e)$. It follows from (Wang and Gao \citeyear{Wang1}) that conditions $(a)-(d)$ are equivalent, and $(e) \Leftrightarrow (f)$ holds due to the invertibility of \(K_1\), which is equivalent to the invertibility of \(K\).\\
\noindent $(b) \Leftrightarrow  (e):$  
We first consider the necessary condition.
From $(b)$, we have 
  $$\operatorname{Rank}(A_s^2)=\operatorname{Rank}(A_s) \quad \text{and} \quad \Sigma_{\mathrm{2d}}=0.$$
  Let $\hat{U} = U_s + \epsilon U_d$, where $U_s,U_d \in \mathbb{C}^{n \times n}$. Since $\hat{U} \in \mathbb{DC}_n^U$, a straightforward calculation shows that $U_s \in \mathbb{C}_n^U$.
Therefore, combined with the invertibility of $\Sigma_{\mathrm{1s}}$, Theorem~\ref{L1} yields
\begin{align*}
\operatorname{Rank}(A_s)&=\operatorname{Rank}\left(U_s\begin{bmatrix}
\Sigma_{\mathrm{1s}} K_{1} & \Sigma_{\mathrm{1s}} L_{1} \\ 0 & 0
\end{bmatrix}U_s^*\right) 
  =\operatorname{Rank}\left(\begin{bmatrix}
\Sigma_{\mathrm{1s}} K_{1} & \Sigma_{\mathrm{1s}} L_{1} \\ 0 & 0
\end{bmatrix}\right) 
=\operatorname{Rank}\left(\begin{bmatrix}
 K_{1} & L_{1} \\ 0 & 0
\end{bmatrix}\right), \\
\operatorname{Rank}(A_s^2) &= \operatorname{Rank}\left(U_s\begin{bmatrix}
(\Sigma_{\mathrm{1s}} K_{1})^2 & \Sigma_{\mathrm{1s}} K_{1}\Sigma_{\mathrm{1s}} L_{1} \\ 0 & 0
\end{bmatrix}U_s^*\right) 
= \operatorname{Rank}\left(\begin{bmatrix}
(\Sigma_{\mathrm{1s}} K_{1})^2 & \Sigma_{\mathrm{1s}} K_{1}\Sigma_{\mathrm{1s}} L_{1} \\ 0 & 0
\end{bmatrix}\right) \\
&=\operatorname{Rank}\left(\begin{bmatrix}
 K_{1}\Sigma_{\mathrm{1s}} K_{1} & K_{1}\Sigma_{\mathrm{1s}} L_{1} \\ 0 & 0
\end{bmatrix}\right).
  \end{align*}
Let $M=\begin{bmatrix} K_{1} & L_{1} \end{bmatrix}$.
From (\ref{S1}), $MM^*=I_r$ holds, which gives $\operatorname{Rank}(M)=r$.
It then follows that 
\begin{align*}
\operatorname{Rank}(A_s^2)=\operatorname{Rank}(A_s) 
&\Leftrightarrow \operatorname{Rank}\left(\begin{bmatrix} K_{1}\Sigma_{\mathrm{1s}} K_{1} & K_{1}\Sigma_{\mathrm{1s}} L_{1}\end{bmatrix} \right)=\operatorname{Rank}\left(\begin{bmatrix} K_{1} & L_{1} \end{bmatrix}\right)\\
&\Leftrightarrow \operatorname{Rank}\left(K_{1}\Sigma_{\mathrm{1s}} \begin{bmatrix} K_{1} & L_{1} \end{bmatrix}\right)=\operatorname{Rank}\left(\begin{bmatrix} K_{1} & L_{1} \end{bmatrix}\right)\\
&\Leftrightarrow \operatorname{Rank}(K_{1}\Sigma_{\mathrm{1s}}M) = \operatorname{Rank}(M)\\
&\Leftrightarrow  \operatorname{Rank}(K_{1}\Sigma_{\mathrm{1s}}M)=r,
\end{align*}
and consequently, $$r=\operatorname{Rank}(K_{1}\Sigma_{\mathrm{1s}}M)\leqslant \operatorname{Rank}(K_{1})\leqslant r.$$
Hence, $\operatorname{Rank}(K_{1}) = r$, which implies that $K_1$ is invertible.\\
Conversely, given $\operatorname{Rank}(K_{1}\Sigma_{\mathrm{1s}}M) = \operatorname{Rank}(M)$,  if $K_1$ is invertible, then
$\operatorname{Rank}(A_s^2)=\operatorname{Rank}(A_s)$ holds naturally, which completes the proof of sufficiency.
\end{proof}

\begin{remark}\label{ANA}
  Let $\hat{A} \in \mathbb{DC}_n^{\mathrm{CM}}$ with $r=\operatorname{Rank}(\hat{A})$. Note that $\Sigma_\mathrm{2d}=0$ is equivalent to $\Sigma_2=0$. From \emph{Lemma~\ref{CS}} and \emph{Theorem~\ref{J1}}, there exists $\hat{U} \in \mathbb{DC}_n^{U}$ such that
  \begin{align} \label{FAN}
 \hat{A} &=\hat{U}\begin{bmatrix}
\Sigma_1 K & \Sigma_1 L\\
0 & 0
\end{bmatrix}\hat{U}^*,
\end{align}
where $K \in \mathbb{DC}^{r \times r}$, $L \in \mathbb{DC}^{r \times (n-r)}$ satisfy $KK^* + LL^* = I_r$,
and $\Sigma_1 = \Sigma_{\mathrm{1s}} + \epsilon \Sigma_{\mathrm{1d}} = \operatorname{diag}(\mu_1, \ldots, \mu_r)$ with $\mu_1 \geq \cdots \geq \mu_r$ being positive appreciable dual numbers.
\end{remark}

\noindent Combining (\ref{E3}) with Theorem~\ref{J1}, we naturally derive the expression for the DGGI.

\begin{theorem}\label{tL1}
  Let $\hat{A} \in \mathbb{DC}_n^{\mathrm{CM}}$ with $r=\operatorname{Rank}(\hat{A})$.
Then there exist $\hat{U} \in \mathbb{DC}_n^{U}$ such that
\begin{align}\label{AGG}
   \hat{A}^{\scalebox{0.5}{\#}}  
  & = \hat{U}\begin{bmatrix}
    K^{-1}\Sigma_1^{-1} &  K^{-1}\Sigma_1^{-1}K^{-1}L \\ 0 & 0
    \end{bmatrix}\hat{U}^*,
 \end{align}   
with $K \in \mathbb{DC}^{r \times r}$ and  $L \in \mathbb{DC}^{r \times (n-r)}$ satisfying $KK^* + LL^* = I_r$,
where $\Sigma_1 = \Sigma_{\mathrm{1s}} + \epsilon \Sigma_{\mathrm{1d}} = \operatorname{diag}(\mu_1, \ldots, \mu_r)$ with $\mu_1 \geq \cdots \geq \mu_r$ being positive appreciable dual numbers.
\end{theorem}
\begin{proof}
The condition  $\hat{A} \in \mathbb{DC}_n^{\mathrm{CM}}$  ensures  that it has the form (\ref{FAN}).
  Let $$ \hat{X}= \hat{U}\begin{bmatrix}
    K^{-1}\Sigma_1^{-1} &  K^{-1}\Sigma_1^{-1}K^{-1}L \\ 0 & 0
    \end{bmatrix}\hat{U}^*.$$
 We first show that $\hat{X}$ satisfies equation (\ref{E3}). Indeed,
\begin{align*}
\hat{A}\hat{X}\hat{A}
&= \hat{U}\begin{bmatrix} \Sigma_1 K & \Sigma_1 L \\ 0 & 0 \end{bmatrix}\hat{U}^*
  \hat{U}\begin{bmatrix} K^{-1}\Sigma_1^{-1} & K^{-1}\Sigma_1^{-1}K^{-1}L \\ 0 & 0 \end{bmatrix}\hat{U}^*
  \hat{U}\begin{bmatrix} \Sigma_1 K & \Sigma_1 L \\ 0 & 0 \end{bmatrix}\hat{U}^* \\
&= \hat{U}\begin{bmatrix} \Sigma_1 K & \Sigma_1 L \\ 0 & 0 \end{bmatrix}\hat{U}^* 
  = \hat{A}; \\[6pt]  
\hat{X}\hat{A}\hat{X}
&= \hat{U}\begin{bmatrix} K^{-1}\Sigma_1^{-1} & K^{-1}\Sigma_1^{-1}K^{-1}L \\ 0 & 0 \end{bmatrix}\hat{U}^*
  \hat{U}\begin{bmatrix} \Sigma_1 K & \Sigma_1 L \\ 0 & 0 \end{bmatrix}\hat{U}^* 
  \hat{U}\begin{bmatrix} K^{-1}\Sigma_1^{-1} & K^{-1}\Sigma_1^{-1}K^{-1}L \\ 0 & 0 \end{bmatrix}\hat{U}^* \\
&= \hat{U}\begin{bmatrix} K^{-1}\Sigma_1^{-1} & K^{-1}\Sigma_1^{-1}K^{-1}L \\ 0 & 0 \end{bmatrix}\hat{U}^*
  = \hat{X}; \\[6pt]
\hat{A}\hat{X}
&= \hat{U}\begin{bmatrix} \Sigma_1 K & \Sigma_1 L \\ 0 & 0 \end{bmatrix}\hat{U}^*
  \hat{U}\begin{bmatrix} K^{-1}\Sigma_1^{-1} & K^{-1}\Sigma_1^{-1}K^{-1}L \\ 0 & 0 \end{bmatrix}\hat{U}^* = \hat{U}\begin{bmatrix} I_r & K^{-1}L \\ 0 & 0 \end{bmatrix}\hat{U}^*=\hat{X}\hat{A}.
\end{align*}
Therefore, $\hat{X}$ is the DGGI of $\hat{A}$.
Since $\hat{A}^{\scalebox{0.5}{\#}}$ exists and is unique, we have $\hat{X}=\hat{A}^{\scalebox{0.5}{\#}}$.
\end{proof}

\begin{theorem}\label{tL}
  Let $\hat{A} \in \mathbb{DC}_n^{\mathrm{CM}}$ with $r=\operatorname{Rank}(\hat{A})$.
Then there exist $\hat{U} \in \mathbb{DC}_n^{U}$ such that
\begin{align}
   \hat{A}^{\scalebox{0.5}{\#}}  
  & = \hat{U}\left(\begin{bmatrix}
    K_1^{-1}\Sigma_1^{-1} &  K_1^{-1}\Sigma_1^{-1}K_1^{-1}L_1 \\ 0 & 0
    \end{bmatrix} + \epsilon \begin{bmatrix}
      -K_1^{-1}K_2K_1^{-1}\Sigma_1^{-1} & R- K_1^{-1}\Sigma_{\mathrm{1s}}^{-2}\Sigma_{\mathrm{1d}}K_1^{-1}L_1 \\ 0 & 0
      \end{bmatrix} \right)\hat{U}^*\notag\\
      &=\hat{U}  \left(\begin{bmatrix}
    K_1^{-1}\Sigma_{\mathrm{1s}}^{-1} &  K_1^{-1}\Sigma_{\mathrm{1s}}^{-1}K_1^{-1}L_1 \\ 0 & 0
    \end{bmatrix} + \epsilon\begin{bmatrix}
      -K_1^{-1}K_2K_1^{-1}\Sigma_{\mathrm{1s}}^{-1}- K_1^{-1}\Sigma_{\mathrm{1s}}^{-2}\Sigma_{\mathrm{1d}} & R \\ 0 & 0
      \end{bmatrix} \right)\hat{U}^*, \label{VB}
\end{align}
where $$R=K_1^{-1}\Sigma_{\mathrm{1s}}^{-1}K_1^{-1}L_2-K_1^{-1}(\Sigma_{\mathrm{1s}}^{-1}K_1^{-1}K_2+K_2K_1^{-1}\Sigma_{\mathrm{1s}}^{-1})K_1^{-1}L_1- K_1^{-1}\Sigma_{\mathrm{1s}}^{-2}\Sigma_{\mathrm{1d}}K_1^{-1}L_1,$$
 $\Sigma_1=\Sigma_{\mathrm{1s}}+\epsilon \Sigma_{\mathrm{1d}}=  \operatorname{diag}(\mu_1, \ldots, \mu_r)$.
Here, $K_1,K_2 \in \mathbb{C}^{r \times r}$, $L_1,L_2 \in \mathbb{C}^{r \times (n-r)}$  satisfy
 $$ K_1 K_1^*+L_1L_1^*=I_r, \quad K_1 K_2^* +K_2 K_1^*+L_1 L_2^*+L_2 L_1^*=0,$$
 and $\mu_1 \geq \cdots \geq \mu_r$ are positive appreciable dual numbers.
\end{theorem}
\begin{proof}
By Theorem~\ref{tL1}, we let
 $$ K=K_1+\epsilon K_2 \in \mathbb{DC}^{r \times r}, \ \  L=L_1+\epsilon L_2 \in \mathbb{DC}^{r \times (n-r)}, \ \ \Sigma_1=\Sigma_{\mathrm{1s}}+\epsilon \Sigma_{\mathrm{1d}},$$
where $K_1,K_2 \in \mathbb{C}^{r \times r}$ and $L_1,L_2 \in \mathbb{C}^{r \times (n-r)}$. 
Thus, as demonstrated by the following calculations,
\begin{align*}
  K^{-1}\Sigma_1^{-1}=&(K_1^{-1}-\epsilon K_1^{-1}K_2K_1^{-1})(\Sigma_{\mathrm{1s}}^{-1}-\epsilon \Sigma_{\mathrm{1s}}^{-2}\Sigma_{\mathrm{1d}})\\
  =& K_1^{-1}\Sigma_{\mathrm{1s}}^{-1} - \epsilon(K_1^{-1}K_2K_1^{-1}\Sigma_{\mathrm{1s}}^{-1}+ K_1^{-1}\Sigma_{\mathrm{1s}}^{-2}\Sigma_{\mathrm{1d}}) ,\\
  K^{-1}\Sigma_1^{-1}K^{-1}L =& K^{-1}\Sigma_1^{-1}(K_1^{-1}-\epsilon K_1^{-1}K_2K_1^{-1})(L_1+\epsilon L_2)\\
  =& [K_1^{-1}\Sigma_{\mathrm{1s}}^{-1} - \epsilon(K_1^{-1}K_2K_1^{-1}\Sigma_{\mathrm{1s}}^{-1}+ K_1^{-1}\Sigma_{\mathrm{1s}}^{-2}\Sigma_{\mathrm{1d}})]
  [K_1^{-1}L_1 + \epsilon (K_1^{-1}L_2-K_1^{-1}K_2K_1^{-1}L_1)]\\
  =& K_1^{-1}\Sigma_{\mathrm{1s}}^{-1}K_1^{-1}L_1 + \epsilon R,
\end{align*}
  where $$R=K_1^{-1}\Sigma_{\mathrm{1s}}^{-1}K_1^{-1}L_2-K_1^{-1}(\Sigma_{\mathrm{1s}}^{-1}K_1^{-1}K_2+K_2K_1^{-1}\Sigma_{\mathrm{1s}}^{-1})K_1^{-1}L_1- K_1^{-1}\Sigma_{\mathrm{1s}}^{-2}\Sigma_{\mathrm{1d}}K_1^{-1}L_1.$$
Then Theorem~\ref{tL1} implies that (\ref{VB}) holds.
\end{proof}

\noindent Based on the D-H-S decomposition, we now derive the expression for the DCGI.

\begin{theorem}\label{tL2}
  Let $\hat{A} \in \mathbb{DC}_n^{\mathrm{CM}}$ with $r=\operatorname{Rank}(\hat{A})$.
Then there exist $\hat{U} \in \mathbb{DC}_n^{U}$ such that
\begin{align}\label{ACG}
   \hat{A}^{\tiny\textcircled{\#}}  
  & = \hat{U}\begin{bmatrix}
    K^{-1}\Sigma_1^{-1} & 0 \\ 0 & 0
    \end{bmatrix}\hat{U}^*,
 \end{align}   
 where $K \in \mathbb{DC}^{r \times r}$, $L \in \mathbb{DC}^{r \times (n-r)}$ satisfy $KK^* + LL^* = I_r$,
and $\Sigma_1 = \Sigma_{\mathrm{1s}} + \epsilon \Sigma_{\mathrm{1d}} = \operatorname{diag}(\mu_1, \ldots, \mu_r)$ with $\mu_1 \geq \cdots \geq \mu_r$ being positive appreciable dual numbers.
\end{theorem}
\begin{proof}
Since $\hat{A} \in \mathbb{DC}_n^{\mathrm{CM}}$, it has the form of (\ref{FAN}).
  Suppose 
  $$ \hat{X}=\hat{U}\begin{bmatrix}
   X_1 & X_2 \\ X_3 & X_4 
\end{bmatrix}\hat{U}^*,$$
where $X_1 \in \mathbb{C}^{r \times r}$, $X_2 \in \mathbb{C}^{r \times (n-r)}$, $X_3\in \mathbb{C}^{(n-r)\times r }$ and $X_4\in \mathbb{C}^{(n-r)\times (n-r) }$.
   Thus,
\begin{align*}
&\hat{A}\hat{X}=\hat{U}\begin{bmatrix}
\Sigma_1(K X_1+L X_3) & \Sigma_1(K X_2+L X_4) \\
0 & 0
\end{bmatrix}\hat{U}^*.
\end{align*}
As $(\hat{A}\hat{X})^*=\hat{A}\hat{X}$, we have
\begin{align}\label{M0}
  KX_2+LX_4=0,
\end{align}
and consequently,
\begin{align*}
&\hat{A}\hat{X}\hat{A}=\hat{U}\begin{bmatrix}
  \Sigma_1(K X_1+L X_3)\Sigma_1K  & \Sigma_1(K X_1+L X_3)\Sigma_1L \\
  0 & 0
\end{bmatrix}\hat{U}^*,\\
 &\hat{A}\hat{X}^2=\hat{U}\begin{bmatrix}
    \Sigma_1(K X_1+L X_3)X_1 & \Sigma_1(K X_1+L X_3)X_2 \\ 0 & 0
    \end{bmatrix}\hat{U}^*.
\end{align*}
Further, $\hat{A}\hat{X}\hat{A}=\hat{A}$ leads to
$$ \Sigma_1(K X_1+L X_3)\Sigma_1K =\Sigma_1K \quad \text{and} \quad \Sigma_1(K X_1+L X_3)\Sigma_1L =\Sigma_1L.$$
Right-multiplying the two equations by $K^*$ and $L^*$ respectively and adding them gives
\begin{align}\label{M1}
  K X_1+L X_3=\Sigma_{1}^{-1}.
\end{align}
Meanwhile, we have $X_3=0$ and  $X_4=0$ by $\hat{A}\hat{X}^2=\hat{X}$.
Thus, from (\ref{M0}) and (\ref{M1}), it can be derived that $X_1=K_1^{-1}\Sigma_1^{-1}$ and $X_2=0$.\\
In conclusion, $$ \hat{X}=\hat{U}\begin{bmatrix}
    K^{-1}\Sigma_1^{-1} & 0 \\ 0 & 0
    \end{bmatrix}\hat{U}^*.$$
The existence and uniqueness of  $\hat{A}^{\tiny\textcircled{\#}}$ as the solution to (\ref{En}) implies that $\hat{X}=\hat{A}^{\tiny\textcircled{\#}}$.
\end{proof}

\begin{theorem}\label{tM}
  Let $\hat{A} \in \mathbb{DC}_n^{\mathrm{CM}}$ with $r=\operatorname{Rank}(\hat{A})$.
Then there exist $U \in \mathbb{DQ}_n^{U}$ such that
\begin{align}
   \hat{A}^{\tiny\textcircled{\#}}  
  &=\hat{U}\left(\begin{bmatrix}
    K_{1}^{-1}\Sigma_1^{-1} & 0 \\ 0 & 0
    \end{bmatrix} - \epsilon\begin{bmatrix}
 K_{1}^{-1}K_{2}K_{1}^{-1}\Sigma_1^{-1} & 0 \\ 0 & 0
 \end{bmatrix} \right)\hat{U}^* \notag\\
 &=\hat{U}\left(\begin{bmatrix}
    K_{1}^{-1}\Sigma_{\mathrm{1s}}^{-1} & 0 \\ 0 & 0
    \end{bmatrix} - \epsilon\begin{bmatrix}
      K_{1}^{-1}K_{2}K_{1}^{-1}\Sigma_{\mathrm{1s}}^{-1}+K_1^{-1}\Sigma_{\mathrm{1s}}^{-2}\Sigma_{\mathrm{1d}} & 0 \\ 0 & 0
 \end{bmatrix} \right)\hat{U}^*, \notag
\end{align}
where $K_1,K_2 \in \mathbb{C}^{r \times r}$, $L_1,L_2 \in \mathbb{C}^{r \times (n-r)}$  satisfy
 $$ K_1 K_1^*+L_1L_1^*=I_r, \quad K_1 K_2^* +K_2 K_1^*+L_1 L_2^*+L_2 L_1^*=0,$$
and $\Sigma_1 = \Sigma_{\mathrm{1s}} + \epsilon \Sigma_{\mathrm{1d}} = \operatorname{diag}(\mu_1, \ldots, \mu_r)$ with $\mu_1 \geq \cdots \geq \mu_r$ being positive appreciable dual numbers.
\end{theorem}
\begin{proof}
We omit the proof, which follows from Theorem~\ref{tL2} after some direct calculation.
\end{proof}

\noindent Theorem~\ref{J1} implies that the DGGI and DCGI exist only if the DMPGI exists. 
Moreover, when they do exist, they can be readily computed using Theorem~\ref{tL1} and Theorem~\ref{tL2}.
The next two examples are provided to illustrate this and to demonstrate the simplicity and ease of computing the DGGI and DCGI using the D-H-S decomposition.

\begin{example}
 Let $$ \hat{A}= \begin{bmatrix}
2+3\epsilon & -\epsilon & -6\epsilon \\
4\epsilon & 3+4\epsilon & 3\epsilon \\
-6\epsilon & 0 & 0
\end{bmatrix},$$
then exists $\hat{U}=\begin{bmatrix}
1 & \epsilon & 3\epsilon \\
-\epsilon & 1 & 0 \\
-3\epsilon & 0 & 1
\end{bmatrix} \in \mathbb{DC}_3^{U}$, 
$\hat{V}=\begin{bmatrix}
1 & 2\epsilon & 3\epsilon \\
-2\epsilon & 1 & -\epsilon \\
-3\epsilon & \epsilon & 1
\end{bmatrix} \in \mathbb{DC}_3^{U}$ such that 
$$ \Sigma = \hat{U}^*\hat{A}\hat{V}= \begin{bmatrix}
2+3\epsilon & 0 & 0 \\
0 & 3 + 4\epsilon & 0 \\
0 & 0 & 0
\end{bmatrix}, \ \ \text{i.e.,}\ \
\Sigma_1=\begin{bmatrix}
  2+3\epsilon & 0 \\
0 & 3 + 4\epsilon 
\end{bmatrix}, \ \ \Sigma_2=0.
$$
Thus, 
$$ \hat{A}=\hat{U}  \begin{bmatrix}
2+3\epsilon & -2\epsilon & 0 \\
3\epsilon & 3+4\epsilon & 3\epsilon \\
0 & 0 & 0
\end{bmatrix}\hat{U}^*,  
$$
where $ K_1=\begin{bmatrix}
  1 & -\epsilon \\ \epsilon & 1 
\end{bmatrix}$ and $ L_1=\begin{bmatrix}
  0 \\ \epsilon
\end{bmatrix}$.
It follows from $\Sigma_2=0$ that the DGGI and DCGI exist. Subsequently,
$\hat{A}^{\scalebox{0.5}{\#}}$ and $\hat{A}^{\tiny\textcircled{\#}}$ are computed using \emph{Theorem~\ref{tL1}} and \emph{Theorem~\ref{tL2}}, yielding:
\begin{align*}
 & \hat{A}^{\scalebox{0.5}{\#}}=\hat{U}\begin{bmatrix}
    K^{-1}\Sigma_1^{-1} &  K^{-1}\Sigma_1^{-1}K^{-1}L \\ 0 & 0
    \end{bmatrix}\hat{U}^*=\hat{U}\begin{bmatrix}
\frac{1}{2}-\frac{3\epsilon}{4} & \frac{\epsilon}{3} & 0 \\
-\frac{\epsilon}{2} & \frac{1}{3}-\frac{4\epsilon}{9} & \frac{\epsilon}{3} \\
0 & 0 & 0
\end{bmatrix}\hat{U}^* =\begin{bmatrix}
\frac{1}{2} - \frac{3\epsilon}{4} & \frac{\epsilon}{6} & -\frac{3\epsilon}{2} \\
-\frac{2\epsilon}{3} & \frac{1}{3} - \frac{4\epsilon}{9} & \frac{\epsilon}{3} \\
-\frac{3\epsilon}{2} & 0 & 0
\end{bmatrix};\\
&\hat{A}^{\tiny\textcircled{\#}} =\hat{U}\begin{bmatrix}
    K^{-1}\Sigma_1^{-1} & 0 \\ 0 & 0
    \end{bmatrix}\hat{U}^*= \hat{U}\begin{bmatrix}
\frac{1}{2}-\frac{3\epsilon}{4} & \frac{\epsilon}{3} & 0 \\
-\frac{\epsilon}{2} & \frac{1}{3}-\frac{4\epsilon}{9} & 0 \\
0 & 0 & 0
\end{bmatrix}\hat{U}^* = \begin{bmatrix}
\frac{1}{2} - \frac{3\epsilon}{4} & \frac{\epsilon}{6} & -\frac{3\epsilon}{2} \\
-\frac{2\epsilon}{3} & \frac{1}{3} - \frac{4\epsilon}{9} & 0 \\
-\frac{3\epsilon}{2} & 0 & 0
\end{bmatrix}.
\end{align*}
\end{example}

\noindent We now explore the connections among different dual generalized inverses by leveraging D-H-S decompositions.

\begin{theorem}\label{y01}
   Let $ \hat{A} \in \mathbb{DC}_n^{\mathrm{CM}}$ have the form \emph{(\ref{FAN})}. Then the following are equivalent:\\
\noindent  $(a)$ $\hat{A}^{N}=\hat{A}^{\scalebox{0.5}{\#}}$; \ \ \ \ 
\noindent  $(b)$ $\hat{A}^{N}=\hat{A}^{\tiny{\textcircled{\#}}} $; \\
\noindent  $(c)$ $\hat{A}^{\scalebox{0.5}{\#}} = \hat{A}^{\tiny{\textcircled{\#}}} $;\ \ \ \ \ 
\noindent  $(d)$ $L=0$.
\end{theorem}
\begin{proof}
  We prove only $(c)\Leftrightarrow (d)$, as the proofs for $(a)\Leftrightarrow (d)$ and $(b)\Leftrightarrow (d)$ are analogous.\\
  \noindent  $(c)\Leftrightarrow (d):$ With sufficiency being clear, we proceed to prove necessity.
  From (\ref{AGG}) and (\ref{ACG}),  the equality $\hat{A}^{\scalebox{0.5}{\#}} = \hat{A}^{\tiny{\textcircled{\#}}} $  implies 
  $$ K^{-1}\Sigma_{1}^{-1}K^{-1}L=0.$$
   Therefore, the invertibility of $K$ and $\Sigma_{1}$ implies $L=0$.
\end{proof}

\noindent Theorem~\ref{y01} provides a generalization of the conclusion in Theorem 2 of (Wang et al. \citeyear{Wang5}), and its proof is conveniently facilitated by the D-H-S decomposition.

\begin{remark}\label{R3}
Specifically, as shown in \emph{(\ref{AGG})} and \emph{(\ref{ACG})}, when $L = 0$, we have
$$
\hat{A}^{\scalebox{0.5}{\#}}=\hat{A}^{\tiny\textcircled{\#}} 
=\hat{U}\begin{bmatrix}
    K^{-1}\Sigma_1^{-1} & 0 \\ 0 & 0
    \end{bmatrix}\hat{U}^*.
$$
\end{remark}

\begin{theorem}\label{yd}
  Let $ \hat{A} \in \mathbb{DC}_n^{\mathrm{CM}}$ have the form \emph{(\ref{FAN})}. Then following statements hold: \\
\noindent  $(a)$ $\hat{A}=\hat{A}^{\scalebox{0.5}{\#}}$ if and only if $(\Sigma_{1} K)^2 = I_r$;\\
\noindent  $(b)$ $\hat{A}=\hat{A}^{\tiny{\textcircled{\#}}} $ if and only if $L= 0, (\Sigma_{1} K)^2=I_r$;\\
\noindent  $(c)$ $\hat{A}^*=\hat{A}^{\scalebox{0.5}{\#}}$ if and only if $L= 0, \Sigma_{1}=I_r$;\\
\noindent  $(d)$ $\hat{A}^*=\hat{A}^{\tiny{\textcircled{\#}}} $ if and only if $L= 0, \Sigma_{1}=I_r$.
\end{theorem}
\begin{proof} 
 We only establish condition $(c)$, since the remaining equivalences are immediate from representations (\ref{FAN}), (\ref{AGG}), and (\ref{ACG}).\\
  \noindent $(c):$ Based on $\hat{A}^*=\hat{A}^{\scalebox{0.5}{\#}}$ from (\ref{FAN}) and (\Ref{AGG}), we have
  $$ L^*\Sigma_1=0 \ \ \text{and} \ \  K^*\Sigma_1=K^{-1}\Sigma_1^{-1}.$$
  Thus, the invertibility of $\Sigma_1$ implies $L=0$. 
  Then, it follows from (\ref{SS}) that $K^{-1}=K^*$, which yields $\Sigma_1^2=I_r$, and hence$\Sigma=I_r$.
\end{proof}

\noindent Next, starting from (\ref{AGG}) and (\ref{ACG}), we apply the D-H-S decomposition of dual generalized inverses to derive the expressions for the composite generalized inverses corresponding to DGGI and DCGI.

\begin{theorem}\label{y6}
  Let $ \hat{A} \in \mathbb{DC}_n^{\mathrm{CM}}$. Then the following statements hold: \\
\noindent  $(a)$ $\hat{A}^{\scalebox{0.5}{\#}}=\hat{A}(\hat{A}^3)^{\dagger}\hat{A}$;\ \ \ \ \ \ \ \ \ 
\noindent  $(b)$ $(\hat{A}^{\scalebox{0.5}{\#}})^{\scalebox{0.5}{\#}}=\hat{A}$;\\
\noindent  $(c)$ $(\hat{A}^{\scalebox{0.5}{\#}})^{\dagger}=\hat{A}^{\dagger}\hat{A}^3\hat{A}^{\dagger} $;\ \ \ \ \ \ \
\noindent  $(d)$ $(\hat{A}^*)^{\scalebox{0.5}{\#}}=(\hat{A}^{\scalebox{0.5}{\#}})^{*}$;\\
\noindent  $(e)$ $\hat{A}\hat{A}^{\scalebox{0.5}{\#}}(\hat{A}^{\scalebox{0.5}{\#}})^{\dagger}=\hat{A}^2\hat{A}^{\dagger}$;\ \ \ \ \
\noindent  $(f)$ $(\hat{A}^{\scalebox{0.5}{\#}})^{\dagger}\hat{A}^{\scalebox{0.5}{\#}}\hat{A}=\hat{A}$;\\
\noindent  $(g)$ $(\hat{A}^{*})^{\scalebox{0.5}{\#}}=(\hat{A}^{\dagger})^{\scalebox{0.5}{\#}}$;\ \ \ \ \ \  \ \ \ \
\noindent  $(h)$ $(\hat{A}^{\dagger})^{\scalebox{0.5}{\#}}=(\hat{A}^*)^{\scalebox{0.5}{\#}}\hat{A}^*\hat{A}\hat{A}^*(\hat{A}^*)^{\scalebox{0.5}{\#}}$.
\end{theorem}
\begin{proof}
By Lemma~\ref{J1}, if  $ \hat{A} \in \mathbb{DC}_n^{\mathrm{CM}}$, then $\Sigma_2=0$ and $K$ is invertible. 
Consequently,  $\hat{A}$ takes the form (\ref{FAN}), and thus $\hat{A}^{\dagger}$ exists.\\
  \noindent $(a):$  By (\ref{AN}), it can be shown that
  $$ (\hat{A}^{3})^{\dagger}=\hat{U}\begin{bmatrix}
    K^*\Sigma_1^{-1}(K^{-1}\Sigma_1^{-1})^2 & 0 \\
L^{*}\Sigma_1^{-1}(K^{-1}\Sigma_1^{-1})^2 & 0
\end{bmatrix}\hat{U}^*.$$
This implies that
\begin{align*}
  \hat{A}(\hat{A}^3)^{\dagger}\hat{A}&=\hat{U}\begin{bmatrix}
\Sigma_1 K & \Sigma_1 L\\
0 & 0
\end{bmatrix}\begin{bmatrix}
  K^*\Sigma_1^{-1}(K^{-1}\Sigma_1^{-1})^2 & 0 \\
L^{*}\Sigma_1^{-1}(K^{-1}\Sigma_1^{-1})^2 & 0
\end{bmatrix}\begin{bmatrix}
\Sigma_1 K & \Sigma_1 L\\
0 & 0
\end{bmatrix}\hat{U}^* \\
&= \hat{U}\begin{bmatrix}
(K^{-1}\Sigma_1^{-1})^2 & 0 \\
0 & 0
\end{bmatrix}\begin{bmatrix}
\Sigma_1 K & \Sigma_1 L \\
0 & 0
\end{bmatrix}\hat{U}^*\\
&=\hat{U}\begin{bmatrix}
    K^{-1}\Sigma_1^{-1} & K^{-1}\Sigma_1^{-1}K^{-1}L  \\ 0 & 0
    \end{bmatrix}\hat{U}^*,
\end{align*}
which together with (\ref{ACG}) gives that $\hat{A}^{\scalebox{0.5}{\#}}=\hat{A}(\hat{A}^3)^{\dagger}\hat{A}$.\\
  \noindent $(b):$  By combining the definition of the equation system in (\ref{E3}) with the results in (\ref{AGG}), the conclusion follows.\\
  \noindent $(c):$ This can be directly demonstrated by using  (\ref{AN}) of $\hat{A}^{\scalebox{0.5}{\#}}$ given in (\ref{AGG}), which takes the form 
   \begin{align*}
    (\hat{A}^{\scalebox{0.5}{\#}})^{\dagger}=\hat{U}\begin{bmatrix}
      K^*K \Sigma_1K & 0 \\ L^*K \Sigma_1K & 0
    \end{bmatrix}\hat{U}^*.
   \end{align*}
   Subsequently, simple calculations using (\ref{MP4}) and (\ref{FAN}) yield $(\hat{A}^{\scalebox{0.5}{\#}})^{\dagger}=\hat{A}^{\dagger}\hat{A}^3\hat{A}^{\dagger} $.\\
  \noindent $(d):$ From (\ref{E3}) and (\ref{FAN}), it follows that 
  \begin{align}\label{DD1}
    (\hat{A}^*)^{\scalebox{0.5}{\#}}=\hat{U}\begin{bmatrix}
      \Sigma_1^{-1}(K^{-1})^* & 0 \\
      L^*(K^{-1})^*\Sigma_1^{-1}(K^{-1})^*  & 0
    \end{bmatrix}\hat{U}^*.
  \end{align}
Combining this with (\ref{ACG}), we obtain $(\hat{A}^*)^{\scalebox{0.5}{\#}}=(\hat{A}^{\scalebox{0.5}{\#}})^{*}$. \\ 
\noindent $(e):$ It follows from $(a)$ and $(c)$ that 
$\hat{A}\hat{A}^{\scalebox{0.5}{\#}}(\hat{A}^{\scalebox{0.5}{\#}})^{\dagger}=\hat{A}\hat{A}(\hat{A}^3)^{\dagger}\hat{A}\hat{A}^{\dagger}\hat{A}^3\hat{A}^{\dagger}=\hat{A}^2(\hat{A}^3)^{\dagger}\hat{A}^3\hat{A}^{\dagger} $.
Furthermore, $(a)$ implies that
$$(\hat{A}^3)^{\dagger}\hat{A}^3 =\hat{U}\begin{bmatrix}
    K^*\Sigma_1^{-1}(K^{-1}\Sigma_1^{-1})^2 & 0 \\
L^{*}\Sigma_1^{-1}(K^{-1}\Sigma_1^{-1})^2 & 0
\end{bmatrix} \begin{bmatrix}
  (\Sigma_1K)^3 &  (\Sigma_1K)^2\Sigma_1L \\
  0 & 0 
\end{bmatrix}\hat{U}^*
=\hat{U}\begin{bmatrix}
  K^*K & K^*L\\
  L^*K & L^*L 
  \end{bmatrix}\hat{U}^*,
$$
from which it can be further deduced that $\hat{A}(\hat{A}^3)^{\dagger}\hat{A}^3=\hat{A}$.
Thus, $$\hat{A}\hat{A}^{\scalebox{0.5}{\#}}(\hat{A}^{\scalebox{0.5}{\#}})^{\dagger}=\hat{A}[\hat{A}(\hat{A}^3)^{\dagger}\hat{A}^3]\hat{A}^{\dagger}=\hat{A}[\hat{A}]\hat{A}^{\dagger}=\hat{A}^2\hat{A}^{\dagger}.$$
\noindent $(f):$ Proof can be obtained similarly to $(e)$.\\
\noindent $(g):$ Substituting (\ref{MP4}) into (\ref{E1}) yields
\begin{align*}
  (\hat{A}^{\dagger})^{\scalebox{0.5}{\#}}= \hat{U}\begin{bmatrix}
      \Sigma_1(K^{-1})^* & 0 \\
      L^*(K^{-1})^*\Sigma_1(K^{-1})^*  & 0
    \end{bmatrix}\hat{U}^*,
\end{align*}
and combining this with (\ref{DD1}) proves $(\hat{A}^{*})^{\scalebox{0.5}{\#}}=(\hat{A}^{\dagger})^{\scalebox{0.5}{\#}}$.\\
\noindent $(h):$ Together with $(g)$, (\ref{FAN}) and (\ref{DD1}), this establishes $(\hat{A}^{\dagger})^{\scalebox{0.5}{\#}}=(\hat{A}^*)^{\scalebox{0.5}{\#}}\hat{A}^*\hat{A}\hat{A}^*(\hat{A}^*)^{\scalebox{0.5}{\#}}$.
\end{proof}

\begin{theorem}\label{y7}
Let $ \hat{A} \in \mathbb{DC}_n^{\mathrm{CM}}$. Then the following statements hold: \\
\noindent  $(a)$ $\hat{A}^{\tiny\textcircled{\#}}=\hat{A}^{\scalebox{0.5}{\#}}\hat{A}\hat{A}^{\dagger}$; \ \ \ \ \ \ \ \ \ \ \ \ \ \ 
\noindent  $(b)$ $(\hat{A}^{\tiny\textcircled{\#}})^{\dagger}=\hat{A}^2\hat{A}^{\dagger}$;\\
\noindent  $(c)$ $(\hat{A}^{\tiny\textcircled{\#}})^{\dagger}=(\hat{A}^{\tiny\textcircled{\#}})^{\scalebox{0.5}{\#}}$;\ \ \ \ \ \ \ \ \ \ \ \ \ 
\noindent  $(d)$ $(\hat{A}^{\tiny\textcircled{\#}})^{\tiny\textcircled{\#}}=\hat{A}^2\hat{A}^{\dagger}$; \\
\noindent  $(e)$ $\hat{A}^{\tiny\textcircled{\#}}\hat{A}=\hat{A}^{\scalebox{0.5}{\#}}\hat{A}$; \ \ \ \ \ \ \ \ \ \ \ \ \ \ \ 
\noindent  $(f)$ $(\hat{A}^{\tiny\textcircled{\#}})^2\hat{A}=\hat{A}^{\scalebox{0.5}{\#}}$;\\
\noindent  $(g)$ $\hat{A}^{\tiny\textcircled{\#}}(\hat{A}^{\tiny\textcircled{\#}})^{\dagger}=(\hat{A}^{\tiny\textcircled{\#}})^{\dagger}\hat{A}^{\tiny\textcircled{\#}}$.
\end{theorem}
\begin{proof}
  Since $ \hat{A} \in \mathbb{DC}_n^{\mathrm{CM}}$, it can be expressed in the form of (\ref{FAN}), from which it follows that $\hat{A}^{\dagger}$ exists.\\
  \noindent $(a):$ It is a direct consequence of (\ref{MP4}), (\ref{FAN}), (\ref{AGG}) and (\ref{ACG}).\\ 
  \noindent $(b):$ Observing (\ref{ACG}) gives 
             \begin{align*}
              (\hat{A}^{\tiny\textcircled{\#}})^{\dagger}=\hat{U}\begin{bmatrix}
                \Sigma_1 K & 0 \\ 0 & 0
              \end{bmatrix}\hat{U}^*.
             \end{align*}
             Then, together with (\ref{MP4}) and (\ref{FAN}), the result follows from a straightforward calculation.\\
  \noindent $(c):$  Combining the definition of the equation system in (\ref{E3}) and using (\ref{ACG}), we obtain 
   $$ (\hat{A}^{\tiny\textcircled{\#}})^{\scalebox{0.5}{\#}}=\hat{U}\begin{bmatrix}
                \Sigma_1 K & 0 \\ 0 & 0
              \end{bmatrix}\hat{U}^*.$$
              Recalling $(b)$, the result follows.\\    
  \noindent $(d):$ Based on  $(a)$, we have $(\hat{A}^{\tiny\textcircled{\#}})^{\tiny\textcircled{\#}}=(\hat{A}^{\tiny\textcircled{\#}})^{\scalebox{0.5}{\#}}\hat{A}^{\tiny\textcircled{\#}}(\hat{A}^{\tiny\textcircled{\#}})^{\dagger}$,
  which together with $(b)$ and $(c)$ leads to 
  $$(\hat{A}^{\tiny\textcircled{\#}})^{\tiny\textcircled{\#}}= (\hat{A}^{\tiny\textcircled{\#}})^{\dagger}\hat{A}^{\tiny\textcircled{\#}}(\hat{A}^{\tiny\textcircled{\#}})^{\dagger}
  =\hat{A}^{\tiny\textcircled{\#}}=\hat{A}^2\hat{A}^{\dagger}.$$
  \noindent $(e):$ Multiplying by $\hat{A}$ on the right in $(a)$ yields $\hat{A}^{\tiny\textcircled{\#}}\hat{A}=\hat{A}^{\scalebox{0.5}{\#}}\hat{A}\hat{A}^{\dagger}\hat{A}=\hat{A}^{\scalebox{0.5}{\#}}\hat{A}$.\\
  \noindent $(f):$ Combining with $(a)$, we obtain
  $$(\hat{A}^{\tiny\textcircled{\#}})^2\hat{A}=\hat{A}^{\scalebox{0.5}{\#}}\hat{A}\hat{A}^{\dagger}\hat{A}^{\scalebox{0.5}{\#}}\hat{A}\hat{A}^{\dagger}\hat{A}
  =\hat{A}^{\scalebox{0.5}{\#}}\hat{A}\hat{A}^{\dagger}\hat{A}^{\scalebox{0.5}{\#}}\hat{A}
  =\hat{A}^{\scalebox{0.5}{\#}}\hat{A}\hat{A}^{\dagger}\hat{A}\hat{A}^{\scalebox{0.5}{\#}}
  =\hat{A}^{\scalebox{0.5}{\#}}\hat{A}\hat{A}^{\scalebox{0.5}{\#}}
  =\hat{A}^{\scalebox{0.5}{\#}}.$$
  \noindent $(g):$ From $(a)$ and $(b)$, it follows that
  \begin{align*}
    \hat{A}^{\tiny\textcircled{\#}}(\hat{A}^{\tiny\textcircled{\#}})^{\dagger}&=\hat{A}^{\scalebox{0.5}{\#}}\hat{A}\hat{A}^{\dagger}\hat{A}^2\hat{A}^{\dagger}=\hat{A}^{\scalebox{0.5}{\#}}\hat{A}^2\hat{A}^{\dagger}=\hat{A}\hat{A}^{\scalebox{0.5}{\#}}\hat{A}\hat{A}^{\dagger}=\hat{A}\hat{A}^{\dagger},\\
(\hat{A}^{\tiny\textcircled{\#}})^{\dagger}\hat{A}^{\tiny\textcircled{\#}}&=\hat{A}^2\hat{A}^{\dagger}\hat{A}^{\scalebox{0.5}{\#}}\hat{A}\hat{A}^{\dagger}=\hat{A}^2\hat{A}^{\dagger}\hat{A}\hat{A}^{\scalebox{0.5}{\#}}\hat{A}^{\dagger}=\hat{A}^2\hat{A}^{\scalebox{0.5}{\#}}\hat{A}^{\dagger}
=\hat{A}\hat{A}^{\scalebox{0.5}{\#}}\hat{A}\hat{A}^{\dagger}=\hat{A}\hat{A}^{\dagger},
  \end{align*}
i.e., $\hat{A}^{\tiny\textcircled{\#}}(\hat{A}^{\tiny\textcircled{\#}})^{\dagger}=(\hat{A}^{\tiny\textcircled{\#}})^{\dagger}\hat{A}^{\tiny\textcircled{\#}}$.
\end{proof}

\noindent It is worth noting that although Theorem~\ref{y7} (a) has been presented in (Wang and Gao \citeyear{Wang1}; Wang et al. \citeyear{Wang5}), we provide an alternative proof based on the D-H-S decomposition, which relies solely on direct multiplication of the relevant expressions. 
The proofs of Theorems~\ref{y6} and \ref{y7} further illustrate the efficiency and conciseness of the D-H-S decomposition in studying dual matrices. \\

\noindent Recently, dual partial orders have attracted increasing research attention, and the D-H-S decomposition can also be applied to their study. 
In what follows, we examine the dual-minus partial order (Gao et al.\citeyear{Gao1}) and the D--core partial order (Sitha and Mosi{\'c} \citeyear{Sitha1}) through this decomposition to demonstrate its applicability.

\begin{definition}\emph{(Gao et al.\citeyear{Gao1}; Sitha and Mosi{\'c} \citeyear{Sitha1})}\label{DDD}
  Let $\hat{A}, \hat{B} \in \mathbb{DC}^{n \times n}$. Then\\
\noindent $(a)$ $\hat{A} \overset{D-\tiny\textcircled{\#}}{\leq} \hat{B}$: $\operatorname{Ind}(\hat{A})=\operatorname{Ind}(\hat{B})=1$, $\hat{A}^{\tiny\textcircled{\#}}\hat{A}=\hat{A}^{\tiny\textcircled{\#}}\hat{B}$ and $\hat{A}\hat{A}^{\tiny\textcircled{\#}}=\hat{B}\hat{A}^{\tiny\textcircled{\#}}$;\\
\noindent $(b)$ $\hat{A} \overset{D}{\leq} \hat{B}$: $A_s \leq B_s$ and the DMPGI $\hat{A}^{\dagger},\hat{B}^{\dagger}$, $(\hat{B} - \hat{A})^{\dagger}$ all exist.
\end{definition}

\noindent Based on (\ref{FAN}) and Definition \ref{DDD}, we derive the corresponding form of \(\hat{B}\) under the D--core partial order using the D-H-S decomposition.
\begin{theorem}\label{4.19}
   Let $\hat{A}, \hat{B} \in \mathbb{DC}_n^{CM}$. 
  Then, when $\hat{A}$  is of the form \emph{(\ref{FAN})}, $\hat{A} \overset{D-\tiny\textcircled{\#}}{\leq} \hat{B}$ if and only if there exists $\hat{U} \in \mathbb{DC}_n^U$ such that
  \begin{align}\label{DDD1}
   \hat{A} =\hat{U}\begin{bmatrix}
\Sigma_1 K & \Sigma_1 L\\
0 & 0
\end{bmatrix}\hat{U}^* \ \ \text{and} \ \
    \hat{B} =\hat{U}\begin{bmatrix}
\Sigma_1 K & \Sigma_1 L\\
0 & P
\end{bmatrix}\hat{U}^*,  
  \end{align}
where $K \in \mathbb{DC}^{r \times r}$, $L \in \mathbb{DC}^{r \times (n-r)}$, $P \in \mathbb{DC}^{(n-r) \times (n-r)}$ satisfy
$$ KK^* + LL^* = I_r, \ \ \operatorname{Rank}(P^2)=\operatorname{Rank}(P),$$
and $\Sigma_1 = \Sigma_{\mathrm{1s}} + \epsilon \Sigma_{\mathrm{1d}} = \operatorname{diag}(\mu_1, \ldots, \mu_r)$ with $\mu_1 \geq \cdots \geq \mu_r$ being positive appreciable dual numbers.
\end{theorem}
\begin{proof}
In light of  (\ref{FAN}), let $$\hat{B} =\hat{U}\begin{bmatrix}
X_1 & X_2\\
X_3 & X_4
\end{bmatrix}\hat{U}^*.$$
Then, the part $(a)$ of Definition \ref{DDD} implies that $\hat{B}$  takes the form given in (\ref{DDD1}).
Furthermore, since  $\hat{U} \in \mathbb{DC}_n^U$ and both $K$ and \(\Sigma_1\) are invertible, we have
\begin{align*}
  \operatorname{Rank}(\hat{B})&=\operatorname{Rank}\left( \hat{U}\begin{bmatrix}
\Sigma_1 K & \Sigma_1 L\\
0 & P
\end{bmatrix}\hat{U}^*\right)=\operatorname{Rank}\left( \begin{bmatrix}\Sigma_1 K & 0\\
0 & P
\end{bmatrix}\right)= \operatorname{Rank}(\Sigma_1 K) +\operatorname{Rank}(P), \\
\operatorname{Rank}(\hat{B}^2)&=\operatorname{Rank}\left( \hat{U}\begin{bmatrix}
(\Sigma_1 K)^2 & \Sigma_1 K\Sigma_1 L+\Sigma_1 LP\\
0 & P^2
\end{bmatrix}\hat{U}^*\right)= \operatorname{Rank}\left( \begin{bmatrix}(\Sigma_1 K)^2 & 0\\
0 & P^2
\end{bmatrix}\right) \\
&=\operatorname{Rank}((\Sigma_1 K)^2) +\operatorname{Rank}(P^2) =\operatorname{Rank}(\Sigma_1 K) +\operatorname{Rank}(P^2).
\end{align*}
In turn, it follows from  $\operatorname{Rank}(\hat{B}^2)=\operatorname{Rank}(\hat{B})$  that $ \operatorname{Rank}(P^2)=\operatorname{Rank}(P)$.
\end{proof}

\noindent For example, Theorem 5 in (Wang et al. \citeyear{Wang5}) was originally proved by means of the Core-EP decomposition. 
In the following, we attempt to give a new proof method based on the D-H-S decomposition.

\begin{theorem}\emph{(Wang et al. \citeyear{Wang5})}\label{4.20}
  Let $\hat{A}, \hat{B} \in \mathbb{DC}_n^{CM}$. 
  Then $\hat{A} \overset{D-\tiny\textcircled{\#}}{\leq} \hat{B} \Rightarrow \hat{A} \overset{D}{\leq} \hat{B}$.
\end{theorem}
\begin{proof}
  By Theorem~\ref{4.19},  $\hat{A} \overset{D-\tiny\textcircled{\#}}{\leq} \hat{B}$ if and only if (\ref{DDD1}) and $ \operatorname{Rank}(P^2)=\operatorname{Rank}(P)$ hold.
  Let $P=P_s+\epsilon P_d$, where $P_s, P_d \in \mathbb{C}^{(n-r) \times (n-r)}$.
  Then, combining Lemma~\ref{TJ1} and Theorem~\ref{J1}, we obtain
  $$ (I_m - P_sP_s^\dagger)P_d(I_n - P_s^\dagger P_s) = 0.$$
It is readily seen from (\ref{DDD1}) that $A_s \leq B_s$, and then suppose
\begin{align*}
 M= \hat{B}-\hat{A}= \hat{U}\begin{bmatrix}
0 & 0\\
0 & P
\end{bmatrix}\hat{U}^* = \hat{U}\left( \begin{bmatrix}
0 & 0\\
0 & P_s
\end{bmatrix}+\epsilon\begin{bmatrix}
0 & 0\\
0 & P_d
\end{bmatrix} \right)\hat{U}^* .
\end{align*}
Thus,  under the condition 
\begin{align*}
  & \ \ \ \ \ (I_m - M_sM_s^\dagger)M_d(I_n - M_s^\dagger M_s)\\ 
  &= 
 \hat{U} \left(
    \begin{bmatrix}
I_r & 0\\
0 & I_{m-r}
\end{bmatrix}-\begin{bmatrix}
0 & 0\\
0 & P_s
\end{bmatrix}\begin{bmatrix}
0 & 0\\
0 & P_s^{\dagger}
\end{bmatrix}
 \right)\begin{bmatrix}
0 & 0\\
0 & P_d
\end{bmatrix}
 \left(
    \begin{bmatrix}
I_r & 0\\
0 & I_{n-r}
\end{bmatrix}-\begin{bmatrix}
0 & 0\\
0 & P_s^{\dagger}
\end{bmatrix}\begin{bmatrix}
0 & 0\\
0 & P_s
\end{bmatrix}
\right)\hat{U}^*\\
&= \hat{U} \left(
\begin{bmatrix}
0 & 0\\
0 & (I_m - P_sP_s^\dagger)P_d(I_n - P_s^\dagger P_s) 
\end{bmatrix}\right)\hat{U}^* =0,
\end{align*}
it follows from Lemma~\ref{TJ1} that $M^{\dagger}$ exists, i.e., $(\hat{B} - \hat{A})^{\dagger}$ exists.
  Furthermore, Theorem~\ref{J1} ensures the existence of $\hat{A}^{\dagger},\hat{B}^{\dagger}$ whenever $\hat{A}, \hat{B} \in \mathbb{DC}_n^{CM}$. 
Then, by Definition \ref{DDD} $(b)$, we have $\hat{A} \overset{D}{\leq} \hat{B}$.
\end{proof}

\noindent From the derivation process of Theorem~\ref{4.19} and Theorem~\ref{4.20}, it can be observed that the calculation through the D-H-S decomposition is more concise and efficient in this case than that through the dual Core-EP decomposition adopted in (Gao et al.\citeyear{Gao1}; Sitha and Mosi{\'c} \citeyear{Sitha1}).
Thus, the D-H-S decomposition provides a practical tool for the study of dual partial orders, and the choice of a suitable decomposition method may be determined by the specific content to be investigated.

\section{Conclusion}\label{5}
This paper introduces two new forms of the D-H-S decomposition and uses them to derive explicit expressions for four types of dual generalized inverses. 
Building on these expressions, we analyze the relationships among these inverses, and further apply the decomposition to investigate composite generalized inverses of dual matrices. 
Moreover, we also verify the applicability of the D-H-S decomposition in the study of dual partial orders.
Overall, D-H-S decomposition provides a concise and effective tool for studying dual generalized inverses and their properties.
The following topics are proposed for further research:
\begin{itemize}
\item Investigate the properties of some special matrix classes using the D-H-S decomposition;
\item Employ the D-H-S decomposition for a further investigation into the generalized inverses of dual matrices.
\end{itemize}
\backmatter

\section*{Declarations}
\textbf{Conflict of interest} \ \  There are no relevant financial or non-financial competing interests.

\bibliography{sn-bibliography}
\end{document}